\documentclass[10pt,english]{smfart}

\usepackage[T1]{fontenc}
\usepackage[english,francais]{babel}
\usepackage{latexsym,amscd}
\usepackage{amsmath,amsfonts,amssymb,mathrsfs}
\usepackage{enumerate,euscript}
\usepackage{amssymb,url,xspace,smfthm}
\input xy
\xyoption{all}

\newcommand{\BibTeX}{{\scshape Bib}\kern-.08em\TeX}
\newcommand{\T}{\S\kern .15em\relax }
\newcommand{\AMS}{$\mathcal{A}$\kern-.1667em\lower.5ex\hbox
        {$\mathcal{M}$}\kern-.125em$\mathcal{S}$}

\DeclareMathOperator{\Hom}{Hom}
\DeclareMathOperator{\Image}{Im}

\DeclareMathOperator{\rang}{rk}

\DeclareMathOperator{\Spec}{Spec}

\DeclareMathOperator{\vol}{vol}
\DeclareMathOperator{\indic}{1\!\!1}

\title{Arithmetic Fujita approximation}
\date{\today}
\author{Huayi Chen}

\address{Universit\'e Paris Diderot --- Paris 7, Institut de math\'ematiques de Jussieu.}
\email{chenhuayi@math.jussieu.fr}

\begin{document}
\def\smfbyname{}

\begin{abstract}
We prove an arithmetic analogue of Fujita's
approximation theorem in Arakelov geometry,
conjectured by Moriwaki, by using slope
method and measures associated to $\mathbb
R$-filtrations.
\end{abstract}

\begin{altabstract}
On d\'emontre un analogue arithm\'etique du
th\'eor\`eme d'approximation de Fujita en
g\'eom\'etrie d'Arakelov --- conjectur\'e par
Moriwaki
--- par la m\'ethode de pentes et les mesures
associ\'ees aux $\mathbb R$-filtrations.
\end{altabstract}
\maketitle

\tableofcontents
\section{Introduction}
Fujita approximation is an approximative
version of Zariski decomposition of
pseudo-effective divisors \cite{Zariski62}
which holds for smooth projective surfaces
but fails in general. Let $X$ be a projective
variety defined over a field $K$ and $L$ be a
{\it big} line bundle on $X$, i.e., the {\it
volume} of $L$, defined as
\[\mathrm{vol}(L):=\limsup_{n\rightarrow\infty}
\frac{\rang_KH^0(X,L^{\otimes n})}{ n^{\dim
X}/(\dim X )!},\] is strictly positive. The
Fujita's approximation theorem asserts that,
for any $\varepsilon>0$, there exists a
projective birational morphism
$\nu:X'\rightarrow X$, an integer $p>0$,
together with a decomposition
$\nu^*(L^{\otimes p})\cong A\otimes E$, where
$A$ is an ample line bundle, $E$ is
effective, such that $p^{-{\dim X
}}\mathrm{vol}(A)\geqslant\mathrm{vol}(L)-\varepsilon$.
This theorem had been proved by Fujita
himself \cite{Fujita94} in characteristic $0$
case, before its generalization to any
characteristic case by Takagi
\cite{Takagi07}. It is the source of many
important results concerning big divisors and
volume function in algebraic geometry
context, such as volume function as a limit,
its log-concavity and differentiability, etc.
We refer readers to \cite[11.4]{LazarsfeldII}
for a survey, see also
\cite{Ein_Las_Mus_Nak_Po05,Ein_Las_Mus_Nak_Po06,Bou_Fav_Mat06,Lazarsfeld_Mustata08}.

The arithmetic analogue of volume function
and the arithmetic bigness in Arakelov
geometry have been introduced by Moriwaki
\cite{Moriwaki00,Moriwaki07}. Let $K$ be a
number field and $\mathcal O_K$ be its
integer ring. Let $\mathscr X$ be a
projective arithmetic variety of total
dimension $d$ over $\Spec\mathcal O_K$. For
any Hermitian line bundle $\overline{\mathscr
L}$ on $\mathscr X$, the {\it arithmetic
volume} of $\overline{\mathscr L}$ is defined
as
\begin{equation}
\label{Equ:arithmetic volume function}
\widehat{\mathrm{vol}}(\overline{\mathscr
L}):=\limsup_{n\rightarrow\infty}
\frac{\widehat{h}^0(\mathscr
X,\overline{\mathscr L}^{\otimes n}
)}{n^{d}/d!},\end{equation} where
\[\widehat{h}^0(\mathscr X,\overline{\mathscr L}^{\otimes n})
:=\log\#\{s\in H^0(\mathscr X,\mathscr
L^{\otimes n
})\,|\,\forall\,\sigma:K\rightarrow\mathbb
C,\, \|s\|_{\sigma,\sup}\leqslant 1\}.\]
Similarly, $\overline{\mathscr L}$ is said to
be {\it arithmetically big} if
$\widehat{\mathrm{vol}}(\overline{\mathscr
L})>0$. In \cite{Moriwaki07,Moriwaki08},
Moriwaki has proved that the arithmetic
volume function is continuous with respect to
$\overline{\mathscr L}$, and admits a unique
continuous extension to
$\widehat{\mathrm{Pic}}(\mathscr X)_{\mathbb
R}$. In \cite{Moriwaki07}, he asked the
following question (Remark 5.9 {\it loc.
cit.}): {\it does the Fujita approximation
hold in the arithmetic case?}

The validity of arithmetic Fujita
approximation has many interesting
consequences. For example, assuming that the
arithmetic Fujita approximation is true, then
by arithmetic Riemann-Roch theorem
\cite{Gillet-Soule,Zhang95}, the right side
of \eqref{Equ:arithmetic volume function} is
actually a limit (see \cite[Remark
4.1]{Moriwaki07}).

In \cite{Chen_bigness}, the author proved
that the sequence which defines the
arithmetic volume function converges. This
gives an affirmative answer to a conjecture
of Moriwaki \cite[Remark 4.1]{Moriwaki07}.
Rather than applying the arithmetic analogue
of Fujita approximation, the proof uses its
classical version on the generic fiber and
then appeals to an earlier work of the author
on the convergence of normalized
Harder-Narasimhan polygons, generalizing the
arithmetic Hilbert-Samuel formula (in an
inexplicit way).

One of the difficulty for establishing
arithmetic Fujita approximation is that, if
$\overline{\mathscr A}$ is a ample Hermitian
line subbundle of $\overline{\mathscr L}$
which approximates well $\overline{\mathscr
A}$, then in general the section algebra of
$A_K$ does not approximate that of $L_K$ at
all. In fact, it approximates only the graded
linear series of $L$ generated by small
sections.

In this article, we prove the conjecture of
Moriwaki on the arithmetic Fujita
approximation by using Bost's slope theory
\cite{BostBour96,Bost2001,BostICM} and the
measures associated to $\mathbb
R$-filtrations \cite{Chen08,Chen_bigness}.
The strategy is similar to that in
\cite{Chen_bigness} except that, instead of
using the geometric Fujita approximation in
its classical form, we apply a recent result
of Lazarsfeld and Musta\c{t}\v{a}
\cite{Lazarsfeld_Mustata08} on a very general
approximation theorem for graded linear
series of a big line bundle on a projective
variety, using the theory of Okounkov bodies
\cite{Okounkov96}. It permits us to
approximate the graded linear series of the
generic fiber generated by small sections.
Another important ingredient in the proof is
the comparison of minimum filtration and
slope filtration (Propositions
\ref{Pro:majoration de Fm par Fs} and
\ref{Pro:inversion comparison} {\it infra}),
which relies on the estimations in
\cite{Banaszczyk95,Bost_Kunnemann} for
invariants of Hermitian vector bundles. By
the interpretation of the arithmetic volume
function by integral with respect to limit of
Harder-Narasimhan measures established in
\cite{Chen_bigness}, we prove that the
arithmetic volumes of these subalgebras
approximate the arithmetic volume of the
Hermitian line bundle, and therefore
establish the arithmetic Fujita
approximation.

Shortly after the first version of this
article had been written, X. Yuan told me
that he was working on the same subject and
has obtained the arithmetic Fujita
approximation independently. He also kindly
sent me his article \cite{Yuan08}, where he has
developed an arithmetic analogue of Okounkov
body, inspired by
\cite{Lazarsfeld_Mustata08}. He has also
obtained the log-concavity of the arithmetic
volume function.

The organization of this article is as
follows. In the second section, we introduce
the notion of approximable graded algebras
and study their asymptotic properties. We
then recall the notion of Borel measures
associated to filtered vector spaces. At the
end of the section, we establish a
convergence result for filtered approximable
algebras. The third section is devoted to a
comparison of filtrations on metrized vector
bundles on a number field, which come
naturally from the arithmetic properties of
these objects. We begin by a reminder on
Bost's slope method. Then we introduce the
$\mathbb R$-indexed minimum filtration and
slope filtration for metrized vector bundles
and compare them. We also compare the
asymptotic behaviour of these two types of
filtrations. In the fourth section, we recall
the theorem of Lazarsfeld and Musta\c{t}\v{a}
on the approximability of certain graded
linear series. We then describe some
approximable graded linear series which come
from the arithmetic of a big Hermitian line
bundle on an arithmetic variety. The main
theorem of the article is established in the
fifth section. We prove that the arithmetic
volume of a big Hermitian line bundle can be
approximated by the arithmetic volume of its
graded linear series of finite type, which
implies the Moriwaki's conjecture. Finally in
the sixth section, we prove that, if a graded
linear series generated by small sections
approximates well a big Hermitian line bundle
$\overline{\mathscr L}$, then it also
approximates well the asymptotic measure of
$\overline{\mathscr L}$ truncated at $0$.

\bigskip

\noindent{\bf Acknowledgement }I would like
to thank J.-B. Bost, S. Boucksom and A.
Chambert-Loir for discussions and for their
valuable remarks. I am also grateful to X.
Yuan and S. Zhang for letter communications.

\section{Approximable algebras and asymptotic measures}

In \cite{Chen08,Chen_bigness}, the author has used
measures associated to filtered vector spaces to study
asymptotic invariants of Hermitian line bundles.
Several convergence results have been established for
graded algebras equipped with $\mathbb R$-filtrations,
either under the finite generating condition on the
underlying graded algebra \cite[Theorem 3.4.3]{Chen08},
or under the geometric condition \cite[Theorem
4.2]{Chen_bigness} that the underlying graded algebra
is the section algebra of a big line bundle. However,
as we shall see later in this article, some graded
algebras coming naturally from the arithmetic do not
satisfy these two conditions. In this section, we
generalize the convergence result to a so-called {\it
approximable graded algebra} case.

\subsection{Approximable graded algebras}

In the study of projective varieties, graded algebras
are natural objects which often appear as graded linear
series of a line bundle. In general, such graded
algebras are not always finitely generated. However, according to approximation theorems
due to Fujita \cite{Fujita94}, Takagi \cite{Takagi07},
Lazarsfeld and Musta\c{t}\v{a}
\cite{Lazarsfeld_Mustata08} etc., they can often be
approximated arbitrarily closely by its graded
subalgebras of finite type. Inspired by
\cite{Lazarsfeld_Mustata08}, we formalize this
observation as a notion. In this section, $K$ denotes
an arbitrary field.

\begin{defi}
Let $B=\bigoplus_{n\geqslant 0}B_n$ be an
integral graded $K$-algebra. We say that $B$
is {\it approximable} if the following
conditions are verified:
\begin{enumerate}[(a)]
\item all vector spaces $B_n$ are finite dimensional and $B_n\neq 0$ for sufficiently large $n$;
\item for any $\varepsilon\in(0,1)$, there exists an
integer $p_0\geqslant 1$ such that, for any integer $p\geqslant p_0$, one has
\[\liminf_{n\rightarrow\infty}\frac{\rang(\mathrm{Im}
(S^nB_p\rightarrow
B_{np}))}{\rang(B_{np})}>1-\varepsilon,\] where
$S^nB_p\rightarrow B_{np}$ is the canonical
homomorphism defined by the algebra structure on $B$.
\end{enumerate}
\end{defi}
\begin{rema}\label{Rem:condition for approximation}
The condition (a) serves to exclude the
degenerate case so that the presentation
becomes simpler. In fact, if an integral
graded algebra $B$ is not concentrated on
$B_0$, then by choosing an integer
$q\geqslant 1$ such that $B_q\neq 0$, we
obtain a new graded algebra
$\bigoplus_{n\geqslant 0}B_{nq}$ which
verifies (a). This new algebra often contains
the information about $B$ which interests us.
\end{rema}

\begin{exem}The following are some examples of approximable graded
algebras.
\begin{enumerate}[1)]
\item If $B$ is an integral graded algebra of finite
type such that $B_n\neq 0$ for sufficiently large $n$,
then it is clearly approximable.
\item Let $X$ be a projective variety over $\Spec
K$ and $L$ be a big line bundle on $X$. Then by
Fujita's approximation theorem, the total graded linear
series $\bigoplus_{n\geqslant 0}H^0(X,L^{\otimes n})$
of $L$ is approximable.
\item More generally, Lazarsfeld and Musta\c{t}\v{a}
have shown that, with the notation of 2), any graded
subalgebra of $\bigoplus_{n\geqslant 0}H^0(X,L^{\otimes
n })$ containing an ample divisor and verifying the
condition (a) above is approximable.
\end{enumerate}
We shall revisit the examples 2) and 3) in \S
\ref{SubSec:reminderon geom fujita}.
\end{exem}

The following properties of approximable
graded algebras are quite similar to
classical results on big line bundles.

\begin{prop}\label{Pro:bigness for graded algebras}
Let $B=\bigoplus_{n\geqslant 0}B_n$ be an
integral graded algebra which is
approximable. Then there exists a constant
$a\in\mathbb N\setminus\{0\}$ such that, for
sufficiently large integer $p$, the algebra
$\bigoplus_{n\geqslant
0}\mathrm{Im}(S^{n}B_p\rightarrow B_{np})$
has Krull dimension $a$. Furthermore, denote
by $d(B):=a-1$. The sequence
\begin{equation}\label{Equ:suite de rang normalise}\Big(\frac{\rang B_n}{n^{d(B)}/d(B)!}\Big)_{n\geqslant 1}\end{equation}
converges in $\mathbb R_+$.
\end{prop}
\begin{proof}
Assume that $B_m\neq 0$ for all $m\geqslant m_0$, where
$m_0\in\mathbb N$. Since $B$ is integral, for any integer $n\geqslant 1$ and any integer $m\geqslant m_0$, one has \begin{equation}
\label{Equ:rang Bn+m majore par rang Bn}
\rang(B_{n+m})\geqslant\rang(B_n).\end{equation}
For any integer  $p\geqslant
m_0$, denote by $a(p)$ the Krull dimension of
$\bigoplus_{n\geqslant
0}\mathrm{Im}(S^{np}B_p\rightarrow B_{np})$, and define
\begin{equation}\label{Equ:fp}
f(p):=\displaystyle\liminf_{n\rightarrow\infty}
\frac{\rang(\mathrm{Im}(S^nB_p\rightarrow
B_{np}))}{\rang B_{np}}.\end{equation}
The approximable condition shows that $\displaystyle\lim_{p\rightarrow\infty}f(p)=1$.
Recall
that the classical result on Hilbert
polynomials implies
\[\rang(\mathrm{Im}(S^nB_p\rightarrow
B_{np}))\asymp n^{a(p)-1}\qquad(n\rightarrow\infty).\]
Thus, if $f(p)>0$, then $\rang B_{np}\asymp
n^{a(p)-1}$, and hence  by \eqref{Equ:rang Bn+m majore par rang Bn}, one has $\rang B_n\asymp n^{a(p)-1}$ ($n\rightarrow\infty$). Thus $a(p)$ is constant if $f(p)>0$. In particular, $a(p)$ is constant when $p$ is sufficiently large. We denote by $a$ this constant,
and by $d(B)=a-1$.

In the following, we shall establish the convergence of the sequence \eqref{Equ:suite de rang normalise}. It suffices to establish
\begin{equation}\label{Equ:liminf plus grand que limsup}\liminf_{n\rightarrow\infty}\frac{\rang
B_n}{n^{d(B)}}
\geqslant\limsup_{n\rightarrow\infty}\frac{\rang
B_n}{n^{d(B)}}.\end{equation} By \eqref{Equ:rang Bn+m majore par rang Bn}, for any integer $p\geqslant 1$, one has
\begin{equation}\label{Equ:comparaison de limsup et liminf}
\limsup_{n\rightarrow\infty}\frac{\rang(B_n)}{n^{d(B)}}
=\limsup_{n\rightarrow\infty}\frac{\rang(
B_{np})}{(np)^{d(B)}}\quad\text{and}\quad\liminf_{n\rightarrow\infty}
\frac{\rang(B_n)}{n^{d(B)}}=\liminf_{n\rightarrow\infty}
\frac{\rang(B_{np})}{(np)^{d(B)}}\end{equation} Suppose
that $f(p)>0$. Then  one has
\[\begin{split}
\limsup_{n\rightarrow\infty}\frac{\rang(B_{np})}{(np)^{d(B)}}
&=\Big(\limsup_{n\rightarrow\infty}\frac{\rang(B_{np})}{\rang(\Image(S^nB_p\rightarrow
B_{np} ))}\Big)\cdot\Big(\lim_{n\rightarrow\infty}
\frac{\rang(\Image(S^nB_p\rightarrow
B_{np}))}{(np)^{d(B)}}\Big)\\
&=f(p)^{-1}\lim_{n\rightarrow\infty}
\frac{\rang(\Image(S^nB_p\rightarrow
B_{np}))}{(np)^{d(B)}}\leqslant
f(p)^{-1}\liminf_{n\rightarrow\infty}
\frac{\rang(B_{np})}{(np)^{d(B)}}.
\end{split}\]
Combining with \eqref{Equ:comparaison de limsup et
liminf} and the approximable hypothesis, we obtain
\eqref{Equ:liminf plus grand que limsup}.
\end{proof}

\begin{coro}\label{Cor:Bn+r sur Bn converge vers 1}
For any $r\in\mathbb N$, one has
\[\lim_{n\rightarrow\infty}\frac{\rang(B_{n+r})}{\rang(B_n)}=1.\]
\end{coro}

\begin{defi}
Let $B$ be an integral graded $K$-algebra which is
approximable. We denote by $\mathrm{vol}(B)$ the limit
\[\mathrm{vol}(B):=\lim_{n\rightarrow\infty}\frac{\rang(B_n)}{n^{d(B)}/d(B)!}.\]
Note that, if $B$ is the total graded linear series of
a big line bundle $L$, then $\mathrm{vol}(B)$ is just
the volume of the line bundle $L$.
\end{defi}

\begin{rema}
It might be interesting to know whether any
approximable graded algebra can always be realized as a
graded linear series of a big line bundle.
\end{rema}

\subsection{Reminder on $\mathbb R$-filtrations}
Let $K$ be a field and $W$ be a vector space of finite
rank over $K$. For {\it filtration} on $W$ we mean a
sequence $\mathcal F=(\mathcal F_tW)_{t\in\mathbb R}$
of vector subspaces of $W$, satisfying the following
conditions:
\begin{enumerate}[1)]
\item if $t\leqslant s$, then $\mathcal F_sW\subset\mathcal
F_tW$;
\item $\mathcal F_tW=0$ for sufficiently positive $t$,
$\mathcal F_tW=W$ for sufficiently negative $t$;
\item the function $t\mapsto\rang(\mathcal F_tW)$ is
left continuous.
\end{enumerate}
The couple $(W,\mathcal F)$ is called a {\it filtered
vector space}.

If $W\neq 0$, we denote by $\nu_{(W,\mathcal F)}$ (or
simply $\nu_W$ if this does not lead to any ambiguity)
the Borel probability measure obtained by taking the
derivative (in the sense of distribution) of the
function $t\mapsto -\rang\mathcal F_tW/\rang W$. If
$W=0$, then there is a unique filtration on $W$ and we
define $\nu_0$ to be the zero measure by convention.
Note that the measure $\nu_W$ is actually a linear
combination of Dirac measures.

All filtered vector spaces and linear maps preserving
filtrations form an exact category. The following
proposition shows that mapping $(W,\mathcal
F)\mapsto\nu_{(W,\mathcal F)}$ behaves well with
respect to exact sequences.
\begin{prop}\label{Pro:suite exacte courte}
Assume that \[\xymatrix{0\ar[r]&(W',\mathcal
F')\ar[r]&(W,\mathcal F)\ar[r]&(W'',\mathcal
F'')\ar[r]&0}\] is an exact sequence of filtered vector
spaces. Then
\[\nu_W=\frac{\rang W'}{\rang W}\nu_{W'}+\frac{\rang W''}{\rang W}
\nu_{W''}.\]
\end{prop}
\begin{proof}
For any $t\in\mathbb R$, one has
\[\rang(\mathcal F_tW)=\rang(\mathcal F'_tW')+\rang(\mathcal F''_tW''),\]
which implies the proposition by taking the
derivative in the sense of distribution.
\end{proof}

\begin{coro}\label{Cor:mesure d'un sous espace}
Let $(W,\mathcal F)$ be a non-zero filtered
vector space, $V\subset W$ be a non-zero
subspace, equipped with the induced
filtration,
$\varepsilon=1-\rang(V)/\rang(W)$. Then for
any bounded Borel function $h$ on $\mathbb
R$, one has
\[\bigg|\int h\,\mathrm{d}\nu_W-\int h\,\mathrm{d}\nu_V\bigg|
\leqslant 2\varepsilon\|h\|_{\sup}.\]
\end{coro}
\begin{proof}
The case where $W=V$ is trivial. In the
following, we assume that $U:=W/V$  is
non-zero, and is equipped with the quotient
filtration. By Proposition \ref{Pro:suite
exacte courte}, one has
\[\nu_W=(1-\varepsilon)\nu_V+\varepsilon\nu_U=\nu_V+
\varepsilon(\nu_U-\nu_V).\] Therefore
\[\bigg|\int h\,\mathrm{d}\nu_W-\int h\,\mathrm{d}\nu_V
\bigg|=\varepsilon\bigg|\int h\,\mathrm{d}\nu_U-\int
h\,\mathrm{d}\nu_V\bigg|\leqslant
2\varepsilon\|h\|_{\sup}.\]
\end{proof}

Let $(W,\mathcal F)$ be a filtered vector space. We
denote by $\lambda:W\rightarrow\mathbb
R\cup\{+\infty\}$ the mapping which sends $x\in W$ to
\[\lambda(x):=\sup\{a\in\mathbb R\,|\,x\in\mathcal F_aW\}.\]
The function $\lambda$ takes values in $\mathrm{supp}
(\nu_W)\cup\{+\infty\}$, and is finite on
$W\setminus\{0\}$. We define
\begin{equation}\label{Equ:lambda}\lambda_{\max}(W)=\max_{x\in
W\setminus\{0\}}\lambda(x)\quad\text{and}\quad
\lambda_{\min}(W)=\min_{x\in
W}\lambda(x).\end{equation} Note that when $W\neq 0$,
one always has $\lambda_{\min}(W)\leqslant\int
x\,\nu_W(\mathrm{d}x)\leqslant\lambda_{\max}(W)$.
However, $\lambda_{\min}(0)=+\infty$ and
$\lambda_{\max}(0)=-\infty$.

We introduce an order ``$\prec$'' on the space
$\mathscr M$ of all Borel probability measures on
$\mathbb R$. Denote by $\nu_1\prec\nu_2$, or
$\nu_2\succ\nu_1$ the relation:
\begin{quote}
\it for any bounded increasing function $h$ on $\mathbb
R$, $\int h\,\mathrm{d}\nu_1\leqslant\int
h\,\mathrm{d}\nu_2$.
\end{quote}
For any $x\in\mathbb R$, denote by $\delta_x$ the Dirac
measure concentrated at $x$. For any $a\in\mathbb R$,
let $\tau_a$ be the operator acting on the space
$\mathscr M$ which sends a measure $\nu$ to its direct
image by the map $x\mapsto x+a$.

\begin{prop}\label{Pro:mesure et comparaison def iltration}
Let $(V,\mathcal F)$ and $(W,\mathcal G)$ be
non-zero filtered vector spaces. Assume that
$\phi:V\rightarrow W$ is an isomorphism of
vector spaces and $a$ is a real number such
that $\phi(\mathcal F_tV)\subset\mathcal
G_{t+a}W$ holds for all $t\in\mathbb R$, or
equivalently, $\forall x\in V$,
$\lambda(\phi(x))\geqslant\lambda(x)+a$, then
$\nu_W\succ\tau_a\nu_V$.
\end{prop}
See \cite[Lemma 1.2.6]{Chen08} for proof.

\subsection{Convergence of measures of an approximable algebra}
\label{Subsec:converges de mesure} Let $B$ be
an integral graded algebra, assumed to be
approximable. Let $f:\mathbb
N\rightarrow\mathbb R_+$ be a mapping. Assume
that, for each integer $n\geqslant 0$, the
vector space $B_n$ is equipped with an
$\mathbb R $-filtration $\mathcal F$ such
that $B$ is $f$-quasi-filtered, that is,
there exists $n_0\in\mathbb N$ such that, for
any integer $l\geqslant 2$, and all
homogeneous elements $x_1,\cdots,x_l$ in $B$
of degrees $n_1,\cdots,n_l$ in $\mathbb
Z_{\geqslant n_0}$, respectively, one has
\[\lambda(x_1\cdots x_l)\geqslant\sum_{i=1}^l\big(\lambda(x_i)-f(n_i)\big),\]
where $\lambda$ is the function defined in
\eqref{Equ:lambda}.

For any $\varepsilon>0$, let $T_\varepsilon$
be the operator acting on the space $\mathscr
M$ of all Borel probability measures which
sends $\nu\in\mathscr M$ to its direct image
by the mapping $x\mapsto\varepsilon x$.

The purpose of this subsection is to establish the
following convergence result, which is a generalization
of \cite[Theorem 4.2]{Chen_bigness}.

\begin{theo}\label{Thm:convergence de mesures}
Let $B$ be an approximable graded algebra
equipped with filtrations as above such that
$B$ is $f$-quasi-filtered. Assume in addition
that
\[\displaystyle\sup_{n\geqslant
1}\lambda_{\max}(B_n)/n<+\infty \quad\text{and}\quad
\displaystyle\lim_{n\rightarrow\infty}f(n)/n=0.\] Then
the sequence $(\lambda_{\max}(B_n)/n)_{n\geqslant 1}$
converges in $\mathbb R$, and the measure sequence
$(T_{\frac 1n}\nu_{B_n})_{n\geqslant 1}$ converges
vaguely to a Borel probability measure on $\mathbb R$.
\end{theo}
\begin{rema}
We say that a sequence $(\nu_n)_{n\geqslant
1}$ of Borel measures on $\mathbb R$
converges vaguely to a Borel measure $\nu$
if, for any continuous function $h$ on
$\mathbb R$ whose support is compact, one has
\[\lim_{n\rightarrow+\infty}\int h\,\mathrm{d}\nu_n
=\int h\,\mathrm{d}\nu.\]
\end{rema}
\begin{proof}
The first assertion has been established in
\cite[Proposition 3.2.4]{Chen08} in a more general
setting without the approximable condition on $B$. Here
we only prove the second assertion.

Assume that $B_n\neq 0$ holds for any
$n\geqslant m_0$, where $m_0\geqslant n_0$ is
an integer, and denote by $\nu_n=T_{\frac
1n}\nu_{B_n}$. The supports of $\nu_n$ are
uniformly bounded from above since
$\sup_{n\geqslant
1}\lambda_{\max}(B_n)/n<+\infty$. Let $p$ be
an integer such that $p\geqslant m_0$. Denote
by $A^{(p)}$ be the graded subalgebra of $B$
generated by $B_p$. For any integer
$n\geqslant 1$, we equipped each vector space
$A^{(p)}_{np}$ with the induced filtration,
and denote by $\nu_n^{(p)}:=T_{\frac
1{np}}\nu_{A^{(p)}_{np}}$. Furthermore, we
choose, for any $r\in\{p+1,\cdots,2p-1\}$, a
non-zero element $e_r\in B_r$, and define
\begin{gather*}
M_{n,r}^{(p)}=e_rB_{np}\subset B_{np+r},\quad
N_{n,r}^{(p)}=e_{3p-r}M_{n,r}^{(p)}\subset
B_{(n+3)p},\\
a_{n,r}^{(p)}=\frac{\lambda(e_{3p-r})-f(np+r)-
f(3p-r)}{np},\quad b_{n,r}^{(p)}=a_{n,r}^{(p)}+
\frac{\lambda(e_r)-f(np)-f(r)}{np},\\
\nu_{n,r}^{(p)}=T_{\frac 1{np}}\nu_{M_{n,r}^{(p)}},
\quad\eta_{n,r}^{(p)}=T_{\frac1{np}}\nu_{N_{n,r}^{(p)}}.
\end{gather*}
Note that, for all $x\in B_{np}$, $y\in M_{n,r}^{(p)}$,
one has
\begin{gather*}
\lambda(e_rx)\geqslant\lambda(x)+\lambda(e_r)-f(np)-f(r),\\
\lambda(e_{3p-r}y)\geqslant\lambda(y)+\lambda(e_{3p-r})-f(3p-r)
-f(np+r).
\end{gather*}
By Proposition \ref{Pro:mesure et comparaison def
iltration}, one has
\[\eta_{n,r}^{(p)}\succ\tau_{a_{n,r}^{(p)}}\nu_{n,r}^{(p)}\succ\tau_{b_{n,r}^{(p)}}
\nu_{np}.\] Let $h(x)$ be a bounded
increasing and continuous function on
$\mathbb R$ whose support is bounded from
below, and which is constant when $x$ is
sufficiently positive. One has
\begin{equation}\label{Equ:encadrement}
\int h\,\mathrm{d}\eta_{n,r}^{(p)}\geqslant \int
h\,\mathrm{d}\tau_{a_{n,r}^{(p)}}\nu_{n,r}^{(p)}\geqslant\int
h\,\mathrm{d}\tau_{b_{n,r}^{(p)}}\nu_{np}.
\end{equation}
Note that $|h(x+\varepsilon x)-h(x)|$ converges
uniformly to zero when $\varepsilon\rightarrow 0$. By
Corollaries \ref{Cor:mesure d'un sous espace} and
\ref{Cor:Bn+r sur Bn converge vers 1}, we obtain
\begin{gather}
\label{Equ:eta et nu} \lim_{n\rightarrow\infty}
\bigg|\int h\,\mathrm{d}\eta_{n,r}^{(p)}-\int
h\,\mathrm{d}\nu_{(n+3)p}\bigg|=0,\\
\label{Equ:nunrp et nu np+r}\lim_{n\rightarrow\infty}
\bigg|\int h\,\mathrm{d}\nu_{n,r}^{(p)}-\int
h\,\mathrm{d}\nu_{np+r}\bigg|=0.
\end{gather}
Note that $|h(x+u)-h(x)|$ converges uniformly to zero
when $u\rightarrow 0$. Combining with the fact that
$\displaystyle\lim_{n\rightarrow\infty}a_{n,r}^{(p)}=\lim_{n\rightarrow\infty}
b_{n,r}^{(p)}=0$, we obtain
\begin{gather}
\label{Equ:tau a}\lim_{n\rightarrow\infty}\bigg|\int
h\,\mathrm{d}\tau_{a_{n,r}^{(p)}}\nu_{n,r}^{(p)}-\int
h\,\mathrm{d}\nu_{n,r}^{(p)}\bigg|=0,\\
\label{Equ:tau b} \lim_{n\rightarrow\infty} \bigg|\int
h\,\mathrm{d}\tau_{b_{n,r}^{(p)}}\nu_{np}-\int
h\,\mathrm{d}\nu_{np}\bigg|=0.
\end{gather}
Thus
\begin{equation}\label{Equ:estimation de limsup}\begin{split}
&\quad\;\limsup_{n\rightarrow\infty}\bigg|\int
h\,\mathrm{d}\nu_{np+r}-\int
h\,\mathrm{d}\nu_{np}|=\limsup_{n\rightarrow\infty}\bigg|\int
h\,\mathrm{d} \tau_{a_{n,r}^{(p)}}\nu_{n,r}^{(p)}-\int
h\,\mathrm{d}
\tau_{b_{n,r}^{(p)}}\nu_{np}\bigg|\\
&\leqslant\limsup_{n\rightarrow\infty}\bigg| \int
h\,\mathrm{d}\eta_{n,r}^{(p)}-\int
h\,\mathrm{d}\tau_{b_{n,r}^{(p)}}\nu_{np}\bigg|=\limsup_{n\rightarrow\infty}\bigg|\int
h\,\mathrm{d}\nu_{(n+3)p}-\int
h\,\mathrm{d}\nu_{np}\bigg|,
\end{split}\end{equation}
where the first equality comes from \eqref{Equ:nunrp et
nu np+r}, \eqref{Equ:tau a} and \eqref{Equ:tau b}. The
inequality comes from \eqref{Equ:encadrement}, and the
second equality results from \eqref{Equ:eta et nu} and
\eqref{Equ:tau b}.

Let $\varepsilon\in (0,1)$. By the
approximability condition on $B$, there
exists two integers $p\geqslant m_0$ and
$n_1\geqslant 1 $ such that, for any integer
$n\geqslant n_1$, one has
\[\frac{\rang A_{np}^{(p)}}{\rang B_{np}}>1-\varepsilon.\]
Therefore, by Corollary \ref{Cor:mesure d'un sous
espace}, one has
\begin{equation}\label{Equ:nu np et nu n pussance p}
\bigg|\int h\,\mathrm{d}\nu_{np}-\int
h\,\mathrm{d}\nu_{n}^{(p)}\bigg|\leqslant
2\varepsilon\|h\|_{\sup}.
\end{equation}
As $A^{(p)}$ is an algebra of finite type, by
\cite[Theorem 3.4.3]{Chen08}, the sequence of
measures $(\nu_n^{(p)})_{n\geqslant 1}$
converges vaguely to a Borel probability
measure $\nu^{(p)}$. Note that the supports
of measures $\nu_n^{p}$ are uniformly bounded
from above. Hence $(\int
h\,\mathrm{d}\nu_n^{(p)})_{n\geqslant 1}$ is
a Cauchy sequence. After the relations
\eqref{Equ:estimation de limsup} and
\eqref{Equ:nu np et nu n pussance p}, we
obtain that, there exists an integer
$n_2\geqslant 1$ such that, for any integers
$m$ and $n$, $m\geqslant n_2$, $n\geqslant
n_2$, one has
\[\bigg|\int h\,\mathrm{d}\nu_n-\int h\,\mathrm{d}\nu_m\bigg|
\leqslant
8\varepsilon\|h\|_{\sup}+\varepsilon.\] Since
$\varepsilon$ is arbitrary, the sequence
$(\int h\,\mathrm{d}\nu_n)_{n\geqslant 1}$
converges in $\mathbb R$. Denote by
$C_0^\infty(\mathbb R)$ the space of all
smooth functions of compact support on
$\mathbb R$. Since any function in
$C_0^\infty(\mathbb R)$ can be written as the
difference of two continuous increasing and
bounded functions whose supports are both
bounded from below which are constant on a
neighbourhood of $+\infty$, we obtain that
\[h\longmapsto\lim_{n\rightarrow\infty} h\,\mathrm{d}\nu_n\]
is a well defined positive continuous linear
functional on $(C_0^{\infty}(\mathbb
R),\|\cdot\|_{\sup})$. As
$C_0^{\infty}(\mathbb R)$ is dense in the
space $C_c(\mathbb R)$ of all continuous
functions of compact support on $\mathbb R$
with respect to the topology induced by
$\|\cdot\|_{\sup}$, the linear functional
extends continuously to a Borel measure $\nu$
on $\mathbb R$. Finally, by Corollary
\ref{Cor:mesure d'un sous espace} and by
passing to the limit, we obtain that for any
$p\geqslant m_0$, one has
\[|1-\nu(\mathbb R)|=|\nu^{(p)}(\mathbb R)-\nu(\mathbb R)|\leqslant 1-f(p),\]
where $f(p)$ was defined in \eqref{Equ:fp}. As
$\displaystyle\lim_{p\rightarrow\infty}f(p)=1$, $\nu$
is a probability measure.
\end{proof}

\section{Comparison of filtrations on metrized vector bundles}

Let $K$ be a number field and $\mathcal O_K$
be its integer ring. Denote by $\delta_K$ the
degree of $K$ over $\mathbb Q$. For {\it
metrized vector bundle} on $\Spec\mathcal
O_K$ we mean a projective $\mathcal
O_K$-module $E$ together with a family
$(\|\cdot\|_\sigma)_{\sigma:K\rightarrow\mathbb
C}$, where $\|\cdot\|_{\sigma}$ is a norm on
$E_{\sigma,\mathbb C}$, assumed to be
invariant by the complex conjugation. We
often use the expression $\overline E$ to
denote the couple
$(E,(\|\cdot\|_\sigma)_{\sigma:K\rightarrow\mathbb
C})$. If for each $\sigma$,
$\|\cdot\|_\sigma$ is a Hermitian norm, we
say that $\overline E$ is a {\it Hermitian
vector bundle}. Any metrized vector bundle of
rank one is necessarily a Hermitian vector
bundle, and we call it a {\it Hermitian line
bundle}.

As pointed out by Gaudron \cite[\S
3]{Gaudron07}, the category of metrized
vector bundles on $\Spec\mathcal O_K$ is
equivalent to that of adelic vector bundles
on $\Spec K$.

On a metrized vector bundle on $\Spec\mathcal O_K$, one
has a natural filtration defined by successive minima.
On a Hermitian vector bundle, there is another
filtration defined by successive slopes. In this
section, we compare these two filtrations.

\subsection{Reminder on the slope
method}\label{SubSec:slope method} In this
section, we recall some notions and results
of Bost's slope method. The references are
\cite{BostBour96,Bost2001,Chambert,BostICM}.

Let $\overline L$ be a Hermitian line bundle
on $\Spec\mathcal O_K$. The {\it Arakelov
degree} of $\overline L$ is defined as
\[\widehat{\mathrm{deg}}(\overline L):=\log\#(L/\mathcal O_Ks)
-\sum_{\sigma:K\rightarrow\mathbb
C}\log\|s\|_\sigma,\] where $s$ is an
arbitrary non-zero element in $L$. This
definition does not depend on the choice of
$s$, thanks to the product formula. An
equivalent definition is
\begin{equation}\label{Equ:degre dArakelov}\widehat{\deg}(\overline L)=-\sum_{\mathfrak
p\in\mathrm{Spm}\mathcal O_K}
\log\|s\|_{\mathfrak
p}-\sum_{\sigma:K\rightarrow\mathbb
C}\log\|s\|_\sigma,\end{equation} this time
$s$ could be an arbitrary element in $L_K$,
$\|\cdot\|_{\mathfrak p}$ is induced by the
$\mathcal O_K$-module structure on $L$. For a
Hermitian vector bundle $\overline E$ of
arbitrary rank, the {\it Arakelov degree} of
$\overline E$ is just
$\widehat{\deg}(\overline
E):=\widehat{\deg}(\Lambda^{\rang E}\overline
E)$, where the metrics of $\Lambda^{\rang
E}\overline E$ are exterior product metrics.
The Arakelov degree of the zero vector bundle
is zero. Furthermore, it is additive with
respect to short exact sequences.

When $\overline E$ is non-zero, the {\it
slope} of $\overline E$ is by definition the
quotient
\[\widehat{\mu}(\overline E):=
\frac{\widehat{\deg}(E)}{\delta_K\rang(E)},\] where
$\delta_K=[K:\mathbb Q]$. As in the case of vector
bundles on curves, the {\it maximal slope}
$\widehat{\mu}_{\max}(\overline E)$ and the {\it
minimal slope} $\widehat{\mu}_{\min}(\overline E)$ of
$\overline E$ are defined as the maximal value of
slopes of all non-zero Hermitian subbundles of
$\overline E$ and the minimal value of all non-zero
Hermitian quotient bundles of $\overline E$,
respectively. The existence of these extremal slopes
are due to Stuhler \cite{Stuhler76} and Grayson
\cite{Grayson84}. One has
$\widehat{\mu}_{\min}(\overline
E)=-\widehat{\mu}_{\max}(\overline E^\vee)$. A non-zero
Hermitian vector bundle is said to be {\it semistable}
if the equality $\widehat{\mu}_{\max}(\overline
E)=\widehat{\mu}(\overline E)$ holds. The results of
Stuhler and of Grayson mentioned above permit to
establish the analogue of Harder-Narasimhan filtration
in Arakelov geometry:
\begin{prop}
There exists a unique flag
\[0=E_0\subsetneq E_1\subsetneq\cdots\subsetneq E_n=E\]
of $E$ such that each subquotient
$\overline{E_i/E_{i-1}}$
($1\in\{1,\cdots,n\}$) is semistable, and
that, by writing
$\mu_i=\widehat{\mu}(\overline{E_i/E_{i-1}})$,
the inequalities of successive slopes
$\mu_1>\mu_2>\cdots>\mu_n$ hold.
\end{prop}

Let $\overline E$ and $\overline F$ be two Hermitian
vector bundles. The {\it height} of a homomorphism
$\phi:E_K\rightarrow F_K$ is defined as the sum of the
logarithms of norms of all local homomorphism induced
from $\phi$ by extending scalars, divided by
$\delta_K$, that is,
\[h(\phi):=\frac{1}{\delta_K}\bigg(
\sum_{\mathfrak p}\log\|\phi_{\mathfrak
p}\|_{\mathfrak
p}+\sum_{\sigma:K\rightarrow\mathbb
C}\log\|\phi\|_\sigma\bigg).\] It is negative
or zero notably when $\phi$ is {\it
effective}, i.e., $\phi$ gives rise to an
$\mathcal O_K$-linear homomorphism, and
$\|\phi_\sigma\|\leqslant 1$ for any
embedding $\sigma:K\rightarrow\mathbb C$.

The following slope inequality compares the
slopes of two Hermitian vector bundles, given
an injective homomorphism between them.

\begin{prop}
Let $\overline E$ and $\overline F$ be two
non-zero Hermitian vector bundles and
$\phi:E_K\rightarrow F_K$ be an injective
$K$-linear homomorphism. Then then following
inequality holds:
\[\widehat{\mu}_{\max}(\overline E)\leqslant\widehat{\mu}_{\max}(\overline F)+h(\phi),\]
where $h(\phi)$ is the height of $\phi$.
\end{prop}

By passing to dual Hermitian vector bundles,
we obtain the following corollary.
\begin{coro}\label{Cor:inegalite des pentes mumin}
Let $\overline E$ and $\overline F$ be two
non-zero Hermitian vector bundles and
$\psi:E_K\rightarrow F_K$ be a surjective
homomorphism. Then
\[\widehat{\mu}_{\min}(\overline F)\geqslant
\widehat{\mu}_{\min}(\overline E)-h(\psi).\]
\end{coro}

\subsection{Minimum filtration and slope filtration}

Let $\overline E$ be a metrized vector bundle
on $\Spec\mathcal O_K$. Let $r$ be the rank
of $E$ and $i\in\{1,\cdots,r\}$. Recall the
$i^{\text{th}}$ (logarithmic) {\it minimum}
of $\overline E$ is defined as
\[e_i(\overline E):=-\log\inf\{a>0\,|\,\rang(\mathrm{Vect}_K\{\mathbb B(\overline E,a)\})\geqslant i\},\]
where $\mathbb B(\overline E,a)=\{s\in
E\,|\,\forall\sigma:K\rightarrow\mathbb
C,\,\|s\|_\sigma\leqslant a\}$. Denote by
$e_{\max}(\overline E)=e_1(\overline E)$ and
$e_{\min}(\overline E)=e_r(\overline E)$.
Define an $\mathbb R$-filtration $\mathcal
F^M$ on $E_K$ as
\[\mathcal F^{M}_tE_K:=\mathrm{Vect}_K\{\mathbb B(
\overline E,e^{-t})\},\] called the {\it minimum
filtration} of $\overline E$.  Note that
$\lambda_{\max}(E_K,\mathcal F^M)=e_{\max}(\overline
E)$ and $\lambda_{\min}(E_K,\mathcal
F^M)=e_{\min}(\overline E)$.

Assume that $\overline E$ is a Hermitian
vector bundle, we define another $\mathbb
R$-filtration $\mathcal F^S$ on $E_K$ such
that \begin{equation}\label{Equ:HN
filtration}\mathcal
F^S_tE_K:=\sum_{\begin{subarray}{c}
F\subset E\\
\widehat{\mu}_{\min}(\overline F)\geqslant t
\end{subarray}}F_K,\end{equation}
where $\widehat{\mu}_{\min}(0)=+\infty$ by
convention. The filtration $\mathcal F^S$ is
called the {\it slope filtration} of
$\overline E$. The slope filtration is just a
reformulation of the Harder-Narasimhan
filtration of $\overline E$ in considering
the successive slopes at the same time. In
particular, if $\overline E_t$ is the
saturated Hermitian vector subbundle of $E$
such that $E_{t,K}=\mathcal F^S_tE_K$, then
one has
\begin{equation}\label{Equ:mu min de Et minore par t}
\widehat{\mu}_{\min}(\overline E_t)\geqslant
t.\end{equation} Moreover, one has
\[\lambda_{\max}(E_K,\mathcal F^S)=\widehat{\mu}_{\max}(\overline
E), \qquad\lambda_{\min}(E_K,\mathcal
F^S)=\widehat{\mu}_{\min}(\overline E).\]
 See \cite[\S
2.2]{Chen08} for details.

\begin{rema}
\'E. Gaudron has generalized the notions of
maximal slope, minimal slope and
Harder-Narasimhan filtration for metrized
vector bundles, cf. \cite[\S 5.2]{Gaudron07}.
However, it is not clear if the $\mathbb
R$-indexed version of his definition of
Harder-Narasimhan filtration (by using
maximal slopes) coincides with \eqref{Equ:HN
filtration} when the metrized vector bundle
is not Hermitian.
\end{rema}

The slope filtration has the following
functorial property, which is an application
of the slope inequality. For proof, see
\cite[Proposition 2.2.4]{Chen08}.
\begin{prop}
Let $\overline E$ and $\overline F$ be two
Hermitian vector bundles on $\Spec\mathcal
O_K$, $\phi:E_K\rightarrow F_K$ be a
homomorphism. Then for any real number $t$,
one has
\begin{equation}\label{Equ:fonctorialte}\phi(\mathcal
F^S_tE_K)\subset \mathcal
F^S_{t-h(\phi)}F_K.\end{equation}
\end{prop}

Let $E$ be a non-zero projective $\mathcal
O_K$-module of finite rank. Let
$g=(\|\cdot\|_\sigma)_{\sigma:K\rightarrow\mathbb
C}$ and
$g'=(\|\cdot\|'_{\sigma})_{\sigma:K\rightarrow\mathbb
C }$ be two families of norms on $E$. We
assume that all metrics $\|\cdot\|'_\sigma$
are Hermitian. Define
\[D(E,g,g'):=\max_{\sigma:K\rightarrow\mathbb C}
\sup_{0\neq s\in E_{\sigma,\mathbb C
}}\big|\log\|s\|_\sigma-\log\|s\|_\sigma'\big|
\]
Denote by $\mathcal F^M$ the minima
filtration of $(E,g)$, and by $\mathcal F^S$
the slope filtration of $(E,g')$.
\begin{prop}\label{Pro:majoration de Fm par Fs}
One has, for any $t\in\mathbb R$,
\[\mathcal F^M_tE_K\subset \mathcal F^S_{t-\alpha}E_K,\]
where $\alpha=\log\sqrt{r}+D(E,g,g')$.
\end{prop}
\begin{proof} Without loss of generality, we
assume that $\mathcal F^M_tE_K\neq 0$. Let
$\overline F$ be the saturation of $\mathcal
F^M_tE_K$ in $E$, equipped with metrics
induced from $g'$. Thus $\overline F$ becomes
a Hermitian vector subbundle of $(E,g')$. Let
$a$ be the rank of $F$. As $F$ is generated
by elements in $\mathbb B((E,g),e^{-t})$,
there exists non-zero elements
$s_1,\cdots,s_a$ in $F$ which form a basis of
$F_K$ and such that $\|s_i\|_\sigma\leqslant
e^{-t}$ for any $\sigma:K\rightarrow\mathbb C
$. One has $\|s_i\|_{\sigma}'\leqslant
e^{-t+D(E,g,g')}$. Let
$\phi:\overline{\mathcal O}_K^{\oplus
a}\rightarrow\overline F$ be the homomorphism
defined by $(s_1,\cdots,s_a)$, where the
Hermitian metrics on $\overline{\mathcal
O}_K$ are trivial. One has
\[h(\phi)\leqslant \log\sqrt{a}-t+D(E,g,g')\leqslant
\log\sqrt{r}-t+D(E,g,g'),\] where $r=\rang
(E) $. By Corollary \ref{Cor:inegalite des
pentes mumin}, the inequality
$\widehat{\mu}_{\min}(\overline F)\geqslant
t-\log\sqrt{r}-D(E,g,g')$ holds. Therefore,
$F_K\subset\mathcal F^S_{t-\alpha}E_K$.
\end{proof}

In order to establish the inverse comparison,
we need some notation. Let $f:\Spec\mathcal
O_K\rightarrow\Spec\mathbb Z$ be the
canonical morphism. For any Hermitian vector
bundle $\overline F$ on $\Spec\mathcal O_K$,
denote by $f_*\overline F$ the Hermitian
vector bundle on $\Spec\mathbb Z$ whose
underlying $\mathbb Z$ module is $F$ and such
that, for any
$s=(s_\sigma)_{\sigma:K\rightarrow\mathbb C
}\in F\otimes_{\mathbb Z}\mathbb
C=\bigoplus_{\sigma:K\rightarrow\mathbb C}
F\otimes_{\mathcal O_K,\sigma}\mathbb C$, one
has
\[\|s\|^2=\sum_{\sigma:K\rightarrow\mathbb C}
\|s_\sigma\|_{\sigma}^2.\]
\begin{prop}
The inequality $e_{\min}(\overline F)\geqslant
e_{\min}(f_*\overline F)$ holds.
\end{prop}
\begin{proof}
Let $s$ be an arbitrary element in $F$. By definition,
one has $\|s\|\geqslant \|s\|_{\sigma}$ for any
$\sigma:K\rightarrow\mathbb C$. Thus, for any $u>0$,
one has $\mathbb B(\overline F,u)\supset\mathbb
B(f_*\overline F,u)$. Furthermore, if $\mathbb
B(f_*\overline F,u)$ generates $F_{\mathbb Q}$ as a
vector space over $\mathbb Q$, it also generates
$F_{K}$ as a vector space over $K$. Therefore,
$e_{\min}(\overline F)\geqslant e_{\min}(f_*F)$.
\end{proof}

Recall some results in \cite{Banaszczyk95}
and \cite{Bost_Kunnemann}.

\begin{prop}
\begin{enumerate}[1)]
\item Let $\overline G$ be a Hermitian vector
bundle on $\Spec\mathbb Z$. Then
\begin{equation}\label{Banaszczyk}e_{\min}(\overline G)+e_{\max}(\overline
G^\vee)
\geqslant-\log(3\rang(G)/2).\end{equation}
\item Let $\overline F$ be a Hermitian vector bundle
on $\Spec\mathcal O_K$. Then
\begin{equation}\label{Bost-Kunnemann}
\widehat{\mu}_{\max}(\overline F)-\frac
12\log(\delta_K\cdot\rang(F))-\frac{\log|\Delta_K|}{2\delta_K}
\leqslant e_{\max}(\overline
F)\leqslant\widehat{\mu}_{\max}(\overline
F)-\frac12\log(\delta_K),\end{equation} where
$\Delta_K$ is the discriminant of $K$.
\end{enumerate}
\end{prop}
\begin{proof}
See \cite[Theorem 3.1 (iii)]{Banaszczyk95} For 1) and
\cite[(3.23), (3.24)]{Bost_Kunnemann} for 2).
\end{proof}

Denote by $\omega_{\mathcal
O_K}:=\Hom_{\mathbb Z}(\mathcal O_K,\mathbb Z
)$ the canonical module of the number field
$K$. Note that the trace map
$\mathrm{tr}_{K/\mathbb Q}$ is a non-zero
element in $\omega_{\mathcal O_K}$. We equip
$\omega_{\mathcal O_K}$ with the norms such
that $\|\mathrm{tr}_{K/\mathbb Q}\|_\sigma=1$
for any $\sigma:K\rightarrow\mathbb Q$. Thus
we obtain a Hermitian line bundle
$\overline{\omega}_{\mathcal O_K}$. The
Arakelov degree of
$\overline{\omega}_{\mathcal O_K}$ is
$\log|\Delta_K|$, where $\Delta_K$ is the
discriminant of $K$. By \cite[Proposition
3.2.2]{Bost_Kunnemann}, for any Hermitian
vector bundle $\overline F$ over
$\Spec\mathcal O_K$, one has a natural
isomorphism
\begin{equation}\label{Equ:image direct}f_*(\overline
F^\vee\otimes\overline{\omega}_{ \mathcal
O_K})\cong (f_*\overline
F)^\vee.\end{equation}

The following lemma compares the logarithmic last
minimum and the minimal slope.
\begin{lemm}\label{Lem:emin minorer par mumin}
Let $\overline F$ be a Hermitian vector
bundle on $\Spec\mathcal O_K$. One has
\[e_{\min}(\overline F)\geqslant\widehat{\mu}_{\min}(\overline
F)-\log|\Delta_K|-\frac{1}{2}\log
\delta_K-\log(3/2)-\log(\rang(F)).\]
\end{lemm}
\begin{proof}
In fact,
\[\begin{split}e_{\min}(f_*\overline F)&\geqslant -e_{\max}((f_*\overline F)^\vee)
-\log(3\delta_K/2)-\log(\rang F)\\&= -e_{\max}(
f_*(\overline F^\vee\otimes\overline{\omega}_{\mathcal
O_K}))-\log(3\delta_K/2)-\log(\rang F)\\
&\geqslant -\widehat{\mu}_{\max}(\overline
F^\vee\otimes\overline{\omega}_{\mathcal
O_K})+\frac
12\log(\delta_K)-\log(3\delta_K/2)-\log(\rang F)\\
&=\widehat{\mu}_{\min}(\overline
F)-\log|\Delta_K|-\frac{1}{2}\log
\delta_K-\log(3/2)-\log(\rang F),
\end{split}
\]
where the two inequalities comes from
\eqref{Banaszczyk} and
\eqref{Bost-Kunnemann}, the first equality
results from \eqref{Equ:image direct}.
\end{proof}

\begin{rema}
The comparison of minima and slopes has been discussed
in \cite{Soule95,Borek05,Gaudron07,Bost_Kunnemann}. Let
$\overline F$ be an arbitrary Hermitian vector bundle
on $\Spec\mathcal O_K$. Up to now, the best upper bound
for
\[\max_{1\leqslant i\leqslant \rang F}\big|
e_i(\overline F)-\mu_i(\overline F)\big|\] is
of order $\rang(F)\log\rang(F)$, where
$\mu_i(\overline F)$ is the $i^{\mathrm{th}}$
slope of $\overline F$ (see \cite[Definition
5.10]{Gaudron07}). It should be interesting
to know if this upper bound can be improved
to be of order $\log\rang(F)$.
\end{rema}

\begin{prop}\label{Pro:inversion comparison}
With the notation of Proposition
\ref{Pro:majoration de Fm par Fs}. One has,
for any $t\in\mathbb R$,
\[\mathcal F_t^SE_K\subset\mathcal F_{t-\beta}^ME_K,\]
where $\beta=D(E,g,g')+\log|\Delta_K| +\frac 12\log
\delta_K+\log(3/2)+\log(\rang E)$.
\end{prop}
\begin{proof}
Let $\overline F$ be the saturated Hermitian
vector subbundle of $(E,g')$ such that $F_K=
\mathcal F_t^SE_K$. By \cite[Proposition
2.2.1]{Chen08} one has
$\widehat{\mu}_{\min}(\overline F)\geqslant
t$. Lemma \ref{Lem:emin minorer par mumin}
implies
\[e_{\min}(\overline F)\geqslant t-\log|\Delta_K|
-\frac 12\log \delta_K-\log(3/2)-\log(\rang F).\]
Denote by $(F,g)$ the metrized vector bundle whose
metrics are induced from $(E,g)$. One has
\[e_{\min}(F,g)\geqslant e_{\min}(\overline F)-D(E,g,g')
\geqslant t-\beta,\] which implies that
$F_K\subset \mathcal F^M_{t-\beta}E_K$.
\end{proof}

\subsection{Comparison of asymptotic measures}

Let $B=\bigoplus_{n\geqslant 0}B_n$ be an approximable
graded algebra. For any integer $r\geqslant 2$ and any
element $\mathbf{n}=(n_i)\in\mathbb N^r$, denote by
\[\phi_{\mathbf{n}}:B_{n_1}\otimes\cdots\otimes B_{n_r}
\longrightarrow B_{n_1+\cdots+n_r}\] the
canonical homomorphism defined by the algebra
structure of $B$.

For each $n\geqslant 0$, let $(\mathscr
B_n,g_n=(\|\cdot\|_\sigma))$ be a metrized vector
bundle on $\Spec\mathcal O_K$ such that $B_n=\mathscr
B_{n,K}$. Let $(\mathscr
B_n,g_n'=(\|\cdot\|'_{\sigma}))$ be a Hermitian vector
bundle on $\Spec\mathcal O_K$. Define
\[D_n:=D(\mathscr B_n,g_n,g_n')=\max_{\sigma:K\rightarrow\mathbb C}
\sup_{0\neq s\in B_{n,\sigma,\mathbb
C}}\big|\log\|s\|_\sigma-\log\|s\|_\sigma'\big|.\]
Denote by $\mathcal F^M$ the minima filtration of
$(\mathscr B_n,g_n)$ and by $\mathcal F^S$ the slope
filtration of $(\mathscr B_n,g_n')$. Let
\[\nu_n^M=T_{\frac 1n}\nu_{(B_n,\mathcal F^M)}
\qquad\text{and}\qquad \nu_n^S=T_{\frac 1n}
\nu_{(B_n,\mathcal F^S)}.\] In this
subsection, we study the asymptotic behaviour
of measure sequences $(\nu^{M}_n)_{n\geqslant
1}$ and $(\nu^S_n)_{n\geqslant 1}$.

\begin{prop}\label{Pro:convergence de mesure slope}
Assume that the following conditions are satisfied:
\begin{enumerate}[(i)]
\item there exists an integer $n_0\geqslant 1$ and a
function $f:\mathbb N\rightarrow\mathbb R_+$
such that $f(n)=o(n)$ ($n\rightarrow\infty$)
and that, for any integer $l\geqslant 2$ and
any element
$\mathbf{n}=(n_i)_{i=1}^l\in\mathbb
Z_{\geqslant n_0}^l$, the height of
$\phi_{\mathbf{n}}$ is bounded from above by
$f(n_1)+\cdots+f(n_l)$;
\item $\displaystyle\sup_{n\geqslant 1}
\widehat{\mu}_{\max}(\mathscr B_n,g_n')/n<+\infty$.
\end{enumerate}
Then the sequence $(\frac
1n\widehat{\mu}_{\max}(\mathscr B_n,g_n'))_{n\geqslant
1 }$ converges in $\mathbb R$, and the sequence of
measures $(\nu^S_n)_{n\geqslant 1}$ converges vaguely
to a Borel probability measure $\nu$ on $\mathbb R$.
\end{prop}
\begin{proof}
For any $n\in\mathbb N$ and any $t\in\mathbb R$, denote
by $\overline{\mathscr B}_{n,t}$ the Hermitian vector
subbundle of $(\mathscr B_n,g_n' )$ such that $\mathscr
B_{n,t,K}=\mathcal F_t^SB_n$. Let $l\geqslant 2 $ be an
integer, $\mathbf{n}=(n_i)_{i=1}^l\in\mathbb
Z_{\geqslant n_0}^l$ and $(t_i)_{i=1}^l\in\mathbb R^l$.
By using the dual form of \cite[Theorem 1.1]{Chen_pm},
one obtains
\[\begin{split}\widehat{\mu}_{\min}(\overline{\mathscr B}_{n_1,t_1}
\otimes\cdots\otimes\overline{\mathscr
B}_{n_l,t_l})& \geqslant\sum_{i=1}^l \big(
\widehat{\mu}_{\min}(\overline{\mathscr
B}_{n_i,t_i})-\log\rang(B_{n_i})\big)\\
&\geqslant\sum_{i=1}^l\big(t_i-\log\rang(B_{n_i})\big).
\end{split}\]
Furthermore, by the assumption (i), the
height of $\phi_{\mathbf{n}}$ is no grater
than $f(n_1)+\cdots+f(n_l)$. Therefore, the
canonical image of $\mathcal
F_{t_1}^SB_{n_1}\otimes \cdots\otimes\mathcal
F^S_{t_l}B_{n_l}$ in $B_{n_1+\cdots+n_l}$
lies in $\mathcal F^S_{t}B_{n_1+\cdots+n_l}$
with
\[t=\sum_{i=1}^l\big(t_i-f(n_i)-\log\rang(B_{n_i})\big).\]
Let $\widetilde f:\mathbb N\rightarrow\mathbb
R_+$ such that $\widetilde
f(n)=f(n)+\log\rang(B_n)$. The argument above
shows that the graded algebra $B$ is
$\widetilde f$-quasi-filtered. Moreover, by
assumption (i) and Proposition
\ref{Pro:bigness for graded algebras}, one
has $\displaystyle\lim_{n\rightarrow+\infty}
\widetilde f(n)/n=0$. By Theorem
\ref{Thm:convergence de mesures}, the
sequence of measures converges vaguely to a
certain Borel probability measure $\nu$.
\end{proof}

\begin{coro}\label{Cor:comparison of limits}
Under the assumption of Proposition
\ref{Pro:convergence de mesure slope}, if
$\displaystyle\lim_{n\rightarrow\infty}D_n/n=0$, then
the sequence $(\frac 1n e_{\max}(\mathscr
B_n,g_n))_{n\geqslant 1}$ converges to
$\displaystyle\lim_{n\rightarrow\infty}\textstyle\frac
1n\widehat{\mu}_{\max}(\mathscr B_n,g_n')$; and the
sequence of measures $(\nu^M_n)_{n\geqslant 1 }$
converges vaguely to $\nu$.
\end{coro}
\begin{proof}
By Propositions \ref{Pro:mesure et
comparaison def iltration},
\ref{Pro:majoration de Fm par Fs} and
\ref{Pro:inversion comparison}, for any
integer $n\geqslant 1$, one has
\[\tau_{\alpha_n}\nu^S_n\succ\nu^M_n\succ\tau_{-\beta_n}\nu^S_n,\]
where
\[\alpha_n=\frac 1{2n}\log\rang(B_n)+\frac{D_n}{n},\quad
\beta_n=\frac{1}{n}\big( \log|\Delta_K|+\frac 12\log
\delta_K-\log(3/2)-\log\rang(B_n)\big).
\]
As
$\displaystyle\lim_{n\rightarrow\infty}\alpha_n=\lim_{n\rightarrow\infty}
\beta_n=0$, the assertions result from Proposition
\ref{Pro:convergence de mesure slope}.
\end{proof}

\begin{rema}\label{Rem:condition de rumely}
The assumptions of Proposition \ref{Pro:convergence de
mesure slope} is fulfilled notably when the following
conditions are satisfied:
\begin{quote}
\it \begin{enumerate}[(a)]\item the $K$-algebra
structure on $B$ gives rise to an $\mathcal
O_K$-algebra structure on $\bigoplus_{n\geqslant
0}\mathscr B_n$;
\item
for any $(m,n)\in\mathbb N^2$, any
$\sigma:K\rightarrow\mathbb C$ and for all $s\in
B_{n,\sigma,\mathbb C}$, $s'\in B_{m,\sigma,\mathbb
C}$, one has
$\|ss'\|_{\sigma}\leqslant\|s\|_\sigma\|s'\|_\sigma$;
\item $e_{\max}(\mathscr B_n,g_n)=O(n)$
($n\rightarrow\infty$).
\end{enumerate}
\end{quote}
Note that, under the conditions (a) and (b) above, the
height of $\phi_{\mathbf{n}}$,
$\mathbf{n}=(n_i)_{i=1}^l\in\mathbb Z_{\geqslant
n_0}^l$, does not exceed $\frac 12
\sum_{i=1}^l\log\rang(B_{n_i}) $. The equivalence of
the condition (c) and the condition (ii) in Proposition
\ref{Pro:convergence de mesure slope} results from
\eqref{Bost-Kunnemann}. See \cite[Remark 4.1.6]{Chen08}
for details. One can also compare the conditions above
with those in \cite[page 12]{Rumely_Lau_Varley}.

In this particular case, the graded algebra
$B$, equipped with minimum filtrations, is
actually $0$-quasi-filtered, where $0$
denotes the constant zero function. So we may
deduce the convergence of $(\frac
1ne_{\max}(\mathscr B_n,g_n))_{n\geqslant 1}$
and $(\nu_n^M)_{n\geqslant 1}$ directly from
Theorem \ref{Thm:convergence de mesures}.
However, as we shall see in the proof of
Proposition \ref{Pro:convergence pour
approximable algebra}, the comparison of
limits established in Corollary
\ref{Cor:comparison of limits} will play an
important role in the study of arithmetic
volume function. So we have chosen an
indirect approach to emphasis this
comparison.
\end{rema}

\section{Approximable graded linear series in arithmetic}

In this section, we recall a result on Fujita
approximation for graded linear series due to
Lazarsfeld and Musta\c{t}\v{a}
\cite{Lazarsfeld_Mustata08}. We then give
several examples of approximable graded
linear series which come naturally from the
arithmetic setting.
\subsection{Reminder on geometric Fujita approximation}
\label{SubSec:reminderon geom fujita} Let $K$
be a field and $X$ be a projective variety
(i.e. integral projective scheme) defined
over $K$. Let $L$ be a big line bundle on
$X$. Denote by $B :=\bigoplus_{n\geqslant
0}H^0(X,L^{\otimes n })$ the graded
$K$-algebra of global sections of tensor
powers of $L$. For {\it graded linear series}
of $L$ we mean a graded sub-$K$-algebra of $B
$. The following definition is borrowed from
\cite{Lazarsfeld_Mustata08}.
\begin{defi}
We say that a graded linear series
$W=\bigoplus_{n\geqslant 0}W_n$ of $L$ {\it
contains an ample divisor} if there exists an
integer $p\geqslant 1$, an ample line bundle
$A$ and an effective line bundle $M$ on $X$,
together with a non-zero section $s\in
H^0(X,M)$, such that $L^{\otimes p }\cong
A\otimes M$, and that the homomorphism of
graded algebras
\[\bigoplus_{n\geqslant 0}H^0(X,A^{\otimes n})\longrightarrow\bigoplus_{n\geqslant 0}
H^0(X,L^{\otimes np})\] induced by $s$
factors through $\bigoplus_{n\geqslant
0}W_{np}$.
\end{defi}

\begin{rema}
In \cite[Definition
2.9]{Lazarsfeld_Mustata08}, this condition
was called the ``condition (C)''. As a big
divisor is always the sum of an ample divisor
and an effective one, the total graded linear
series $B $ contains an ample divisor.
\end{rema}

\begin{defi}
Let $W=\bigoplus_{n\geqslant 0}W_n$ be a
graded linear series of $L$. Denote by
$\mathrm{vol}(W )$ the number
\begin{equation}
\mathrm{vol}(W
):=\limsup_{n\rightarrow\infty}\frac{\rang(W_n)}{n^{\mathrm{dim}X}/(\mathrm{dim}
X )!}.
\end{equation}
\end{defi}

Note that $\mathrm{vol}(B )=\mathrm{vol}(L)$.
For a general linear series $W $ of $L$, one
has $\mathrm{vol}(W
)\leqslant\mathrm{vol}(L)$. By using the
method of Okounkov bodies introduced in
\cite{Okounkov96}, Lazarsfeld and
Musta\c{t}\v{a} have established the
following generalization of Fujita's
approximation theorem.

\begin{theo}[Lazarsfeld-Musta\c{t}\v{a}]
\label{Thm:Lazarsfeld-Mustata} Assume that
$W=\bigoplus_{n\geqslant 0}W_n$ is a graded
linear series of $L$ which contains an ample
divisor and such that $W_n\neq 0$ for
sufficiently large $n$. Then $W$ is
approximable.
\end{theo}

In particular, the total graded linear series
$B$ is approximable. In \cite[Remark
3.4]{Lazarsfeld_Mustata08}, the authors have
explained why their theorem implies the
Fujita's approximation theorem in its
classical form. We include their explanation
as the corollary below.

\begin{coro}[Geometric Fujita
approximation]\label{Cor:geometric fujita
appr} For any $\varepsilon>0$, there exists
an integer $p\geqslant 1$, a birational
projective morphism $\varphi:X'\rightarrow
X$, an ample line bundle $A$ and an effective
line bundle $M$ such that
\begin{enumerate}[1)]
\item one has $\varphi^*(L^{\otimes p})\cong A\otimes M$;
\item $\mathrm{vol}(A)\geqslant p^{\dim
X}(\vol(L)-\varepsilon)$.
\end{enumerate}
\end{coro}
\begin{proof}
For any integer $p$ such that $B_p\neq 0$,
let $\varphi_p:X_p\rightarrow X$ be the
blow-up (twisted by $L$) of $X$ along the
base locus of $B_p$. That is
\[X_p=\mathrm{Proj}\Big(\mathrm{Im}
\Big(\bigoplus_{n\geqslant
0}S^n(\pi^*B_p)\longrightarrow\bigoplus_{n\geqslant
0 }L^{\otimes{np}}\Big)\Big).\] Denote by
$E_p$ the exceptional divisor and by $s$ the
global section of $\mathcal O(E_p)$ which
trivializes $\mathcal O(E_p)$ outside the
exceptional divisor. By definition, one has
$\mathcal
O_{X_p}(1)\cong\varphi_p^*L^{\otimes p
}\otimes\mathcal O(-E_p)$. On the other hand,
the canonical homomorphism
$\varphi_p^*\pi^*B_p\rightarrow\mathcal
O_{X_p}(1)$ is surjective, therefore
corresponds to a morphism of schemes
$i_p:X_p\rightarrow\mathbb P(B_p)$ such that
$i_p^*\mathcal O_{\mathbb P(B_p)}(1)=\mathcal
O_{X_p}(1) $. The restriction of global
sections of $\mathcal O_{\mathbb P(B_p)}(n)$
on $X_p$ gives an injective homomorphism
\[\mathrm{Im}(S^nB_p\rightarrow B_{np})
\longrightarrow H^0(X_p,\mathcal
O_{X_p}(n)),\] where we have identified
$H^0(X_p,\mathcal O_{X_p}(n))$ with a
subspace of $H^0(X_p,\varphi_p^*L^{\otimes n
})$ via $s$. Since the total grade linear
series $B$ is approximable, one has
\[\sup_p\liminf_{n\rightarrow\infty}\frac{\rang(\mathrm{Im}(S^nB_p\rightarrow B_{np}))}{
\rang B_{np}}=1,\] which implies
\[\sup_p\lim_{n\rightarrow\infty}\frac{\rang H^0(X_p,\mathcal O_{X_p}(n))}{
(np)^d/d!}=\mathrm{vol}(L).\]

The line bundle $\mathcal O_{X_p}(1)$
constructed above is actually nef and big.
However, a slight perturbation of $L$ permits
to conclude.
\end{proof}

\subsection{Arithmetic volume of approximable graded linear series}
In the sequel, $K$ denotes a number field and $\mathcal
O_K$ denotes its integer ring. Let
$\delta_K:=[K:\mathbb Q]$ be the degree of $K$ over
$\mathbb Q$. Let $\pi:\mathscr
X\rightarrow\Spec\mathcal O_K$ be a projective
arithmetic variety of total dimension $d$ and
$X=\mathscr X_K$. Let $\overline{\mathscr L}$ be a
Hermitian line bundle on $\mathscr X$, supposed to be
big in the sense of Moriwaki \cite{Moriwaki00}. Let
$L=\mathscr L_K$. Note that $L$ is a big line bundle on
$X$.

Let $B$ be a graded linear series of $L$. For any
integer $n\geqslant 0$, denote by $\mathscr B_n$ the
saturation of $B_n$ in $\pi_*(\mathscr L^{\otimes n})$.
For any embedding $\sigma:K\rightarrow\mathbb C$,
denote by $\|\cdot\|_{\sigma,\sup}$ the sup-norm on
$B_{n,\sigma,\mathbb C}$. Thus we obtain a metrized
vector bundle $(\mathscr B_n,g_n)$ with
$g_n=(\|\cdot\|_{\sigma,\sup})_{\sigma:K\rightarrow\mathbb
C }$.

Inspired by \cite{Moriwaki00}, we define the arithmetic
volume function of $B$ as follows:
\[\widehat{\mathrm{vol}}(B):=\limsup_{n\rightarrow\infty}\frac{\widehat{h}^0(\mathscr B_n,g_n)}{
n^d/d!},\] where for any metrized vector bundle
$\overline E=(E,(\|\cdot\|_\sigma))$ on $\Spec\mathcal
O_K$, $\widehat{h}^0(\overline E)$ is defined as
\[\widehat{h}^0(\overline E):=\log\#\{s\in E\,|\,
\forall\sigma:K\rightarrow\mathbb
C,\,\|s\|_\sigma\leqslant 1\}.\]

\begin{prop}\label{Pro:convergence pour approximable algebra}
Assume that the graded linear series $B$ is
approximable. Then the sequence
$(\frac{1}{n}e_{\max}(\mathscr
B_n,g_n))_{n\geqslant 1}$ converges in
$\mathbb R$. Furthermore, for any integer
$n\geqslant 1$, let $\nu_n:=T_{\frac
1n}\nu_{(B_n,\mathcal F^M)}$ be the
normalized probability measure associated to
the minimum filtration of $(\mathscr
B_n,g_n)$, then the sequence of measures
$(\nu_n)_{n\geqslant 1}$ converges vaguely to
a Borel probability measure $\nu_B$.
Moreover, one has
\begin{equation}\label{Equ:integral x+ est volume}\int_{\mathbb
R}\max\{x,0\}\,\nu_B(\mathrm{d}x)=\frac{\widehat{\mathrm{vol}}
(B)}{\delta_Kd\,\mathrm{vol}(B)
}.\end{equation}
\end{prop}
\begin{proof}
To establish the convergence of
$(\frac{1}{n}e_{\max}(\mathscr
B_n,g_n))_{n\geqslant 1}$ and
$(\nu_n)_{n\geqslant 1 }$, it suffices to
prove that $(\mathscr B_n,g_n)$ verify the
conditions in Remark \ref{Rem:condition de
rumely}, where (a) and (b) are
straightforward. In order to prove the
condition (c), we introduce, for any integer
$n\geqslant 1$, an auxiliary family
$g_n'=(\|\cdot\|_\sigma)_{\sigma:K\rightarrow\mathbb
C}$ of Hermitian norms on $\pi_*(\mathscr
L^{\otimes n })$, invariant under complex
conjugation, and such that, for any $0\neq
s\in H^0(X_\sigma(\mathbb C
),L_{\sigma,\mathbb C})$.
\begin{equation}\label{Equ:auxilliary hermitian metric}
\log\|s\|_{\sigma}-\frac
32\log(\rang\pi_*(\mathscr L^{\otimes n
}))\leqslant \log\|s\|_{\sigma,\sup}
\leqslant\log\|s\|_{\sigma}-\frac12\log(\rang
\pi_*(\mathscr L^{\otimes n})).\end{equation}
This is always possible by the argument of
the ellipsoids of John or L\"owner, see
\cite[\S 2]{Gaudron07}. It suffices to
establish the estimation
$\widehat{\mu}_{\max}(\pi_*(\mathscr
L^{\otimes n }),g_n')\ll n$. Let $\Sigma$ be
a generic family (i.e., $\Sigma$ is dense in
$X$) of algebraic points in $X$. Each point
$P$ in $\Sigma$ extends in a unique way to a
$\mathcal O_{K(P)}$ point of $\mathscr X$,
where $K(P)$ is the field of definition of
$P$. Therefore we may consider elements in
$\Sigma$ as points of $\mathscr X$ valued in
algebraic integer rings. Now consider the
evaluation map $\pi_*(\mathscr L^{\otimes
n})\longrightarrow\bigoplus_{P\in\Sigma}P^*\mathscr
L$. It is generically injective since
$\Sigma$ is dense in $X$. Therefore, there
exists a subset $\Sigma_n$ of $\Sigma$ whose
cardinal is $\rang(\pi_*(\mathscr L^{\otimes
n}))$ and such that the evaluation map
\[\phi_n:\pi_*(\mathscr L^{\otimes
n})\longrightarrow\bigoplus_{P\in\Sigma_n}P^*\mathscr
L\]is still generically injective. Therefore, after
suitable extension of the ground field, the slope
inequality asserts that
\[\widehat{\mu}_{\max}(\pi_*(\mathscr L^{\otimes n}),g_n')
\leqslant
\sup_{P\in\Sigma_n}nh_{\overline{\mathscr
L}}(P)+h(\phi_n)\leqslant n
\sup_{P\in\Sigma}h_{\overline{\mathscr
L}}(P).\] Since $\Sigma$ is arbitrary, we
obtain that $\frac
1n\widehat{\mu}_{\max}(\pi_*(\mathscr
L^{\otimes n}),g_n')$ is bounded from above
by the essential minimum of
$\overline{\mathscr L}$ (see \cite[\S 5
]{Zhang95} for definition. Attention, in
\cite{Zhang95}, the author denoted it as
$e_1(\overline{\mathscr L})$).

The equality \eqref{Equ:integral x+ est volume} comes
from the following Lemma.
\begin{lemm}\label{Lem:comparaison des h0}
Let $(E,g=(\|\cdot\|_\sigma))$ be a metrized vector
bundle and $(E,g'=(\|\cdot\|_{\sigma}'))$ be a
Hermitian vector bundle on $\Spec\mathcal O_K$. Assume
that $r:=\rang(E)>0$. Let
\[D=\max_{\sigma:K\rightarrow\mathbb C}\sup_{0\neq s\in E_{\sigma,\mathbb
C}} \big|\log\|s\|_{\sigma}-\log\|s\|_\sigma'\big|.\]
Denote by $\nu$ the Borel probability measure
associated to the Harder-Narasimhan filtration of
$\overline E:=(E,g')$. Then there exists a function
$C_0:\mathbb N_*\rightarrow\mathbb R_+$, independent of
all data above, satisfying $C_0(n)\ll n\log n$, and
such that
\[\bigg|\delta_K r\int_{\mathbb R}\max\{x,0\}\,\nu(\mathrm{d}x)
-\widehat{h}^0(E,g)\bigg|\leqslant
(\delta_KD+\log|\Delta_K|)r+C_0(r).
\]
\end{lemm}
\begin{proof}[Proof of the Lemma]
Denote by $\overline M=(\mathcal
O_K,(\|\cdot\|_{\sigma}^M))$ the Hermitian line bundle
on $\Spec\mathcal O_K$ such that
$\|\mathbf{1}\|_{\sigma}^M=e^{-D}$, where $\mathbf{1}$
is the unit element in $\mathcal O_K$. By d\'efinition,
one has $\widehat{h}^0(\overline E\otimes\overline
M^\vee)\leqslant\widehat{h}^0(\overline
E,g)\leqslant\widehat{h}^0(\overline E\otimes\overline
M )$. Moreover, the Borel probability measures
associated to the Harder-Narasimhan filtrations of
$\overline E\otimes\overline M$ and $\overline
E\otimes\overline M^\vee $ are respectively
$\tau_{D}\nu$ and $\tau_{-D}\nu$. By \cite[Lemma 7.1
and Proposition 3.3]{Chen_bigness}, there exists a
function $C_0:\mathbb N_*\rightarrow\mathbb R_+$,
independent of $\overline E$, satisfying the estimation
$C_n(n)\ll n\log n$, and such that
\begin{gather*}\widehat{h}^0(\overline E\otimes\overline M^\vee)
\geqslant \delta_Kr\int_{\mathbb
R}\max\{x,0\}\,\tau_{-D}\nu(\mathrm{d}x)-r\log|\Delta_K|-C_0(r)\\
\widehat{h}^0(\overline E\otimes\overline M)\leqslant
\delta_Kr\int_{\mathbb
R}\max\{x,0\}\,\tau_D\nu(\mathrm{d}x)+
r\log|\Delta_K|+C_0(r)
\end{gather*}
Since $\max\{x+D,0\}\leqslant\max\{x,0\}+D$ and
$\max\{x-D,0\}\geqslant\max\{x,0\}-D$, we obtain the
desired inequality.
\end{proof}

By using Lemma \ref{Lem:comparaison des h0}, we obtain
\[\bigg|\widehat{h}^0(\mathscr B_n,g_n)-nr_n\delta_K\int_{\mathbb
R}\max\{x,0\}\,\nu_n'(\mathrm{d}x)\bigg|\leqslant
\frac{3}{2}\delta_Kr_n\log(r_n)+r_n\log|\Delta_K|+C_0(r_n),
\]
where $r_n=\rang(B_n)$, and $\nu_n'$ is the Borel
probability measure associated to $(\mathscr B_n,g_n')$
(here we still use $g_n'$ to denote the metrics on
$\mathscr B_n$ induced from $(\pi_*(\mathscr L^{\otimes
n }),g_n')$). We have shown that $(\nu_n')_{n\geqslant
1 }$ also converge vaguely to $\nu_B$. Furthermore,
Proposition \ref{Pro:bigness for graded algebras} show
that $r_n=\mathrm{vol}(B)n^{d-1}/(d-1)!+o(n^{d-1})$. By
passing to limit, we obtain \eqref{Equ:integral x+ est
volume}.
\end{proof}

\subsection{Examples of approximable graded linear series}

In this subsection, we give some examples of
approximable graded linear series of $L$ which come
from the arithmetic.

Denote by $B =\bigoplus_{n\geqslant
0}H^0(X,L^{\otimes n })$ the sectional
algebra of $L$. For any real number
$\lambda$, let $B^{[\lambda]}$ be the graded
sub-$K$-module of $B$ defined as follows:
\begin{equation}\label{Equ:Blambda}B^{[\lambda]}_0:=K,\qquad B^{[\lambda]}_n:=
\mathrm{Vect}_K\big(\{s\in
B_n\,|\,\forall\,\sigma:K\rightarrow\mathbb
C,\,\|s\|_{\sigma,\sup}\leqslant e^{-\lambda
n}\}\big).\end{equation} The following property is
straightforward from the definition.
\begin{prop}
For any $\lambda\in\mathbb R$,
$B^{[\lambda]}$ is a graded linear series of
$L$.
\end{prop}
Note that $B^{[0]}$ is noting but the graded
linear series generated by effective
sections. For any integer $n\geqslant 0$ and
any real number $\lambda$, denote by
$\mathscr B_n=\pi_*(\mathscr L^{\otimes n})$
and by $\mathscr B_n^{[\lambda]}$ the
saturation of $B_n^{[\lambda]}$ in $\mathscr
B_n$. We shall use the symbol $g_n$ to denote
the family of sup-norms on $\mathscr B_n$ or
on $\mathscr B_n^{[\lambda]}$. By definition,
for any integer $n\geqslant 1$ and any
$\lambda\in\mathbb R$, one has
\[B_n^{[\lambda]}=\mathcal F^M_{n\lambda}B_n,\]
where $\mathcal F^M$ is the minimum
filtration of $(\mathscr B_n,g_n)$.

 Since we have
assumed $\overline{\mathscr L}$ to be
arithmetically big, the line bundle $L$ is
also big (see \cite[Introduction]{Moriwaki07}
and \cite[Corollary 2.4]{Yuan07}). Hence by
Theorem \ref{Thm:Lazarsfeld-Mustata}, the
total graded linear series $B$ is
approximable. By Corollary 3.13, we obtain
that the sequence $(\frac 1n
e_{\max}(\mathscr B_n,g_n))_{n\geqslant 1}$
converges to a real number which we denote by
$\widehat{\mu}^{\pi}_{\max}(\overline{\mathscr
L})$. Note that, if $\overline M$ is a
Hermitian line bundle on $\Spec\mathcal O_K$,
then
\[\widehat{\mu}^{\pi}_{\max}(\overline{\mathscr
L }\otimes \pi^*(\overline
M))=\widehat{\mu}^{\pi}_{\max}(\overline{\mathscr
L })+\delta_K^{-1}\widehat{\deg}(\overline
M).\]

For any real number $\lambda$, denote by
$\overline{\mathcal O }_\lambda$ the
Hermitian line bundle on $\Spec\mathcal O_K$
whose underlying $\mathcal O_K$-module is
trivial, and such that
$\|\mathbf{1}\|_{\sigma}=e^{-\lambda}$ for
any $\sigma$. Note that the Arakelov degree
of $\overline{\mathcal O}_\lambda$ is
$\widehat{\deg}(\overline{\mathcal
O}_\lambda)=\delta_K\lambda$.

\begin{prop}
\label{Pro:total linear series is
approximable} Let $\lambda$ be a real number
such that
$\lambda<\widehat{\mu}^\pi_{\max}(\overline{\mathscr
L })$. Then the graded linear series
$B^{[\lambda]}$ contains an ample divisor,
and for sufficiently large $n$, one has
$B_n^{[\lambda]}\neq 0$.
\end{prop}
\begin{proof}
Note that
$\widehat{\mu}^{\pi}_{\max}(\overline{\mathscr
L}\otimes\pi^*\overline{\mathcal
O}_{-\lambda})>0$. Since $L$ is big, by
\cite[Theorem 5.4]{Chen_bigness},
$\overline{\mathscr
L}\otimes\pi^*\overline{\mathcal
O}_{-\lambda}$ is arithmetically big.
Therefore, for sufficiently large $n$,
$\mathscr L^{\otimes n}$ has a non-zero
global section $s_n$ such that
$\|s_n\|_{\sigma,\sup}\leqslant e^{-\lambda
n}$ for any $\sigma:K\rightarrow\mathbb C$,
which proves that $B_n^{[\lambda]}\neq 0$.
Furthermore, since $\overline{\mathscr
L}\otimes\pi^*\overline{\mathcal
O}_{-\lambda}$ is arithmetically big, by
\cite[Corollary 2.4]{Yuan07}, there exists an
integer $p\geqslant 1$ and two Hermitian line
bundles $\overline{\mathscr A}$ and
$\overline{\mathscr M}$, such that
$\overline{\mathscr A}$ is ample in the sense
of Zhang \cite{Zhang95}, $\overline{\mathscr
M}$ has a non-zero effective global section
$s$, and that $(\overline{\mathscr
L}\otimes\pi^*\overline{\mathcal
O}_{-\lambda})^{\otimes
p}\cong\overline{\mathscr
A}\otimes\overline{\mathscr M}$. By taking
$p$ sufficiently divisible, we may assume
that the graded $K$-algebra
$\bigoplus_{n\geqslant 0}H^0(X,\mathscr
A_K^{\otimes n})$ is generated by effective
sections of $\mathscr A $. These sections,
viewed as sections of $\overline{\mathscr
A}\otimes\pi^*\overline{\mathcal
O}_\lambda^{\otimes p}$, have sup-norms
$\leqslant e^{-p\lambda}$. Therefore the
homomorphism
\[\bigoplus_{n\geqslant 0}H^0(X,\mathscr A_K)\longrightarrow
\bigoplus_{n\geqslant 0}H^0(X,L^{\otimes
np})\] induced by $s$ factors through
$\bigoplus_{n\geqslant 0}B_{np}^{[\lambda]}$.
\end{proof}

\begin{coro}\label{Cor:approximablility of mu lambda}
For any real number $\lambda$ such that
$\lambda<\widehat{\mu}_{\max}^\pi(\overline{\mathscr
L} )$, the graded linear series
$B^{[\lambda]}$ of $L$ is approximable.
\end{coro}
\begin{proof}
This is a direct consequence of Proposition
\ref{Pro:total linear series is approximable}
and Theorem \ref{Thm:Lazarsfeld-Mustata}.
\end{proof}

\section{Arithmetic Fujita approximation}
\label{Sec:arithmetic fujita} In this
section, we establish the conjecture of
Moriwaki on the arithmetic analogue of Fujita
approximation. Let $\pi:\mathscr
X\rightarrow\Spec\mathcal O_K$ be an
arithmetic variety of total dimension $d$ and
$\overline{\mathscr L}$ be a Hermitian line
bundle on $\mathscr X$ which is
arithmetically big.

Write $L=\mathscr L_K$ and denote by
$B:=\bigoplus_{n\geqslant 0}H^0(X,L^{\otimes
n })$ the total graded linear series of $L$.
For any integer $n\geqslant 1$, let
$\overline{\mathscr B}_n$ be the $\mathcal
O_K$-module $\pi_*(\mathscr L^{\otimes n})$
equipped with sup-norms. Define by convention
$\overline{\mathscr B}_0$ as the trivial
Hermitian line bundle on $\Spec\mathcal O_K$.
Denote by
$\widehat{\mu}_{\max}^\pi(\overline{\mathscr
L
})=\displaystyle\lim_{n\rightarrow\infty}\textstyle\frac
1n e_{\max}(\overline{\mathscr B}_n)$.

For any real number $\lambda$, let
$B^{[\lambda]}$ be the graded linear series
of $L$ defined in \eqref{Equ:Blambda}. For
any integer $n\geqslant 0$, let
$\overline{\mathscr B}_n^{[\lambda]}$ be the
saturation of $B_n^{[\lambda]}$ in $B_n$
equipped with induced metrics. For any
integer $p\geqslant 1$ such that
$B^{[0]}_p\neq 0$, let $B^{(p)}$ be the
graded sub-$K$-algebra of $B$ generated by
$B^{[0]}_p$. For any integer $n\geqslant 1$,
let $\overline{\mathscr B}_{np}^{(p)}$ be the
saturated Hermitian vector subbundle of
$\overline{\mathscr B}_{np}$ such that
${\mathscr B}^{(p)}_{np,K}=B^{(p)}_{np}$.
\begin{theo}\label{The:main theorem}
The following equality holds:
\[\widehat{\mathrm{vol}}(\overline{\mathscr L})=\sup_{p}\widehat{\mathrm{vol}}\big(B^{(p)}\big),\]
where $B^{(p)}$ is the graded linear series
of $L$ generated by $B_p^{[0]}$ defined
above.
\end{theo}
\begin{proof}
For any integer $n\geqslant 1$, let
$\nu_n=T_{\frac 1n}\nu_{(B_n,\mathcal F^M)}$,
where $\mathcal F^M$ is the minimum
filtration of $\overline{\mathscr B}_n$. We
have shown in Proposition
\ref{Pro:convergence pour approximable
algebra} that the sequence
$(\nu_n)_{n\geqslant 1}$ converges vaguely to
a Borel probability measure which we denote
by $\nu$. Similarly, for any integer
$n\geqslant 1$, let $\nu_n^{(p)}=T_{\frac
1{np}}\nu_{(B_{np}^{(p)},\mathcal F^M)}$. The
sequence $(\nu_n^{(p)})_{n\geqslant 1}$ also
converges vaguely to a Borel probability
measure which we denote by $\nu^{(p)}$.

For any subdivision
$D:0=t_0<t_1<\cdots<t_m<\widehat{\mu}^{\pi}_{\max}(\overline{\mathscr
L })$ of the interval
$[0,\widehat{\mu}^\pi_{\max}(\overline{\mathscr
L})[$ such that
\begin{equation}\label{Equ:condition on subdivision}\nu(\{t_1,\cdots,t_m\})=0,\end{equation}
denote by $h_D:\mathbb R\rightarrow\mathbb R$
the function such that
\[h_D(x)=\sum_{i=0}^{m-1}t_{i}\indic_{[t_{i},t_{i+1}[}(x)
+t_{m}\indic_{[t_{m},\infty[}(x).\] After
Corollary \ref{Cor:approximablility of mu
lambda}, for any $\varepsilon>0$, there
exists a sufficiently large integer
$p=p(\varepsilon,D)\geqslant 1$ such that
$B^{(p)}$ approximates simultaneously all
algebras $B^{[t_i]}$ ($i\in\{0,\cdots,m\}$).
That is, there exists $N_0\in\mathbb N$ such
that, for any $n\geqslant N_0$, one has
\[\inf_{0\leqslant i\leqslant m}\frac{\rang\big(\Image(S^nB^{[t_i]}_p\rightarrow
B^{[t_i]}_{np})\big)}{\rang(B_{np}^{[t_i]})}\geqslant
1-\varepsilon.
\]
We then obtain that
\[\rang(\mathcal F^{M}_{npt_i}B_{np}^{(p)})
\geqslant\rang\big(\Image(S^nB^{[t_i]}_p\rightarrow
B^{[t_i]}_{np})\big)\geqslant
(1-\varepsilon)\rang(B_{np}^{[t_i]}).\] Note
that
\[\begin{split}&\quad\;
np\rang(B_{np}^{(p)})\int_{\mathbb
R}\max\{t,0\}\,\nu_{np}^{(p)}(\mathrm{d}t)=
-\int_{\mathbb R}\max\{
t,0\}\,\mathrm{d}\rang(\mathcal
F^{M}_{t}B_{np}^{(p)})\\
&\geqslant\sum_{i=0}^{m-1}npt_i
\Big(\rang(\mathcal
F^{M}_{npt_i}B_{np}^{(p)})-\rang(\mathcal
F^{M}_{npt_{i+1}}B_{np}^{(p)})\Big)
+npt_m\rang( \mathcal
F^{M}_{npt_m}B_{np}^{(p)}).
\end{split}
\]
By Abel summation formula, one obtains
\[\begin{split}
&\rang(B_{np}^{(p)})\int_{\mathbb R}
\max\{t,0\}\,\nu_{np}^{(p)}(\mathrm{d}t)
\geqslant
\sum_{i=1}^{m}(t_i-t_{i-1})\rang(\mathcal
F^{M}_{npt_i}B_{np}^{(p)})\\
&\qquad\geqslant
(1-\varepsilon)\sum_{i=1}^{m}(t_i-t_{i-1})
\rang(B_{np}^{[t_i]}).\end{split}\] Still by
Abel summation formula, one gets
\[\rang(B_{np}^{(p)})\int_{\mathbb
R}\max\{t,0\}\,
\nu_{np}^{(p)}(\mathrm{d}t)\geqslant
(1-\varepsilon)\rang_{K}(B_{np})\int
h_D(x)\,\nu_{np}(\mathrm{d}t).\] By
\eqref{Equ:integral x+ est volume}, one has
\[\begin{split}
&\quad\;\lim_{n\rightarrow\infty}\frac{\delta_Kd}{(np)^{d-1}/(d-1)!}
\rang(B_{np}^{(p)})\int_{\mathbb
R}\max\{t,0\}\nu_{np}^{(p)}(\mathrm{d}t)\\
&=\delta_Kd\mathrm{vol}(B^{(p)})\int_{\mathbb
R
}\max\{t,0\}\,\nu^{(p)}(\mathrm{d}t)=\widehat{\mathrm{vol}}(B^{(p)}).
\end{split}\]
Therefore,\[
\begin{split}
\widehat{\mathrm{vol}}(B^{(p)})&\geqslant
\lim_{n\rightarrow\infty}\frac{\delta_K
d}{(np)^{d-1}/(d-1)!}(1-\varepsilon)\rang(B_{np})
\int_{\mathbb
R}h_D\,\mathrm{d}\nu_{np}\\
&= \delta_Kd(1-\varepsilon)
\mathrm{vol}(L)\int_{\mathbb R}
h_D\,\mathrm{d}\nu,
\end{split}
\]
where the equality follows from \cite[IV \S5
$n^\circ$12 Proposition 22]{Bourbaki65}.
Choose a sequence of subdivisions
$(D_j)_{j\in\mathbb N }$ verifying the
condition \eqref{Equ:condition on
subdivision} and such that $h_{D_j}(t)$
converges uniformly to
$\max\{t,0\}-\max\{t-\widehat{\mu}_{\max}^\pi(\overline{\mathscr
L }),0\}$ when $j\rightarrow\infty$, one
obtains
\[\widehat{\mathrm{vol}}(B^{(p)})\geqslant \delta_Kd(1-\varepsilon)
\mathrm{vol}(L)\int_{\mathbb R}
\max\{t,0\}\,\nu(\mathrm{d}t)=(1-\varepsilon)\widehat{\mathrm{vol}}
(\overline{\mathscr L}),\] thanks to
\eqref{Equ:integral x+ est volume}. The
theorem is thus proved.
\end{proof}

In the following, we explain why Theorem
\ref{The:main theorem} implies the Fujita's
arithmetic approximation theorem in the form
conjectured by Moriwaki. Our strategy is
quite similar to Corollary \ref{Cor:geometric
fujita appr}, except that the choice of
metrics on the approximating invertible sheaf
requires rather subtle analysis on the
superadditivity of probability measures
associated to a filtered graded algebra,
which we put in the appendix.

\begin{theo}[Arithmetic Fujita approximation]
For any $\varepsilon>0$, there exists a
birational morphism $\nu:\mathscr
X'\rightarrow\mathscr X$, an integer
$p\geqslant 1$ together with an decomposition
$\nu^*\overline{\mathscr L}^{\otimes p
}\cong\overline{\mathscr
A}\otimes\overline{\mathscr M}$ such that
\begin{enumerate}[1)]
\item $\overline{\mathscr M}$ is effective and
$\overline{\mathscr A}$ is arithmetically
ample;
\item one has $p^{-d}\widehat{\mathrm{vol}}(\overline{\mathscr
A
})\geqslant\widehat{\mathrm{vol}}(\overline{\mathscr
L })-\varepsilon$.
\end{enumerate}
\end{theo}
\begin{proof} By \cite[Theorem
4.3]{Moriwaki07}, we may assume that
$\mathscr X$ is generically smooth. For any
integer $p\geqslant 1$ such that
$B_p^{[0]}\neq 0$, let $\phi_p:\mathscr
X_p\rightarrow\mathscr X$ be the blow up
(twisted by $\mathscr L$) of $\mathscr X$
along the base locus of $\mathscr B_p^{[0]}$.
In other words, $\mathscr X_p$ is defined as
\[\mathscr X_p=\mathrm{Proj}\Big(\Image\Big(
\bigoplus_{n\geqslant 0}\pi^*\mathscr
B_p^{(np)}\longrightarrow\bigoplus_{n\geqslant
0 }\mathscr L^{\otimes np}\Big)\Big).\] Let
$\mathscr A_p=\mathcal O_{\mathscr X_p}(1)$
and $\mathscr M_p$ be the invertible sheaf
defined by the exceptional divisor. Let $s$
be the global section of $\mathscr M_p$ which
trivializes $\mathscr M_p$ outside the
exceptional divisor. By definition, one has
$\phi_p^*\mathscr L^{\otimes p}\cong\mathscr
A_p\otimes\mathscr M_p $. On the other hand,
the canonical homomorphism
$\phi_p^*\pi^*\mathscr B_p^{[0]}\rightarrow
\mathscr A_p$ induces a morphism
$i_p:\mathscr X_p\rightarrow\mathbb
P(\mathscr B_p^{[0]})$ such that
$i_p^*(\mathcal L_p)\cong\mathscr A_p$, where
$\mathcal L_p=\mathcal O_{\mathbb P(\mathscr
B_p^{[0]})}(1)$. The restriction of global
sections of $\mathcal L_p^{\otimes n}$ gives
an injective homomorphism
\begin{equation}\label{Equ:Image dans H0}\mathrm{Im}(S^n\mathscr
B_p^{[0]}\rightarrow \mathscr
B_{np})=\mathscr B^{(p)}_{np} \longrightarrow
H^0(\mathscr X_p,\mathscr A_p^{\otimes
n}),\end{equation} where the last $\mathcal
O_K$-module is considered as a submodule of
$H^0(\mathscr X_p,\phi_p^*\mathscr L^{\otimes
p })$ via $s$.

For any integer $n\geqslant 1$ and any
embedding $\sigma:K\rightarrow\mathbb C$,
denote by $\|\cdot\|_{\sigma,n}$ the quotient
Hermitian norm on $\mathscr A_{p,\sigma}$
induced by the surjective homomorphism
$\phi_p^*\pi^*\mathscr
B_{np}^{(p)}\rightarrow\mathscr A_p^{\otimes
n }$, where on $\mathscr B_{np}^{(n)}$ we
have chosen the John norm
$\|\cdot\|_{\sigma,\mathrm{John}}$ associated
to the sup-norm $\|\cdot\|_{\sigma,\sup}$
(see \cite[\S 4.2]{Gaudron07} for details).
Thus the Hermitian norms on $\mathscr A_p$
are positive and smooth. Now let
$\sigma:K\rightarrow\mathbb C$ be an
embedding and $x$ be a complex point of
$\mathscr X_p$ outside the exceptional
divisor. It corresponds to an one-dimensional
quotient of $B_{p,\sigma}^{[0]}$, which
induces, for any integer $n\geqslant 1$, an
one-dimensional quotient $l_{n,x}$ of
$B_{np,\sigma}^{(p)}$. By classical result on
convex bodies in Banach space, there exists
an affine hyperplane parallel to the
$\mathrm{Ker}(B_{np,\sigma}^{(p)}\rightarrow
l_{n,x})$ and tangent to the closed unit ball
of $B_{np,\sigma}^{(p)}$. In other words,
there exists $v\in B_{np,\sigma}^{(p)}$ whose
image in $\mathscr A_{p,\sigma}^{\otimes
n}(x)$ has norm
$\|v\|_{\sigma,\mathrm{John}}\geqslant\|v\|_{\sigma,\sup}$.
Note that, as a section of $L_\sigma^{\otimes
n}$ over $X_\sigma(\mathbb C )$, one has
$\|v_x\|_\sigma\leqslant
\|v\|_{\sigma,\sup}$. Hence, for any section
$u$ of $\mathscr A_{p,\sigma}$ over a
neighbourhood of $x$, one has
$\|u_x\|_{\sigma,n}\geqslant\|u_x\otimes
s_x\|_{\sigma}$. Therefore, if we equip
$\mathscr A_p$ with metrics
$\alpha_n=(\|\cdot\|_{\sigma,n})_{\sigma:K\rightarrow\mathbb
C}$ and define $(\mathscr
M_p,\beta_n):=\phi_p^*\overline{\mathscr
L}\otimes(\mathscr A_p,\alpha_n)^\vee$. Then
the section $s$ of $\mathscr M_p$ is an
effective section. For any integer
$n\geqslant 1$, one has
\[p^{-d}\widehat{\vol}(\mathscr
A_p,\alpha_n)\geqslant\widehat{\vol}(B^{(p)},\alpha_n).\]
Note that, for any
$\sigma:K\rightarrow\mathbb C$ and any
element $v\in B_{np,\sigma}^{(p)}$ considered
as a section in $H^0(\mathscr
X_{p,\sigma}(\mathbb C),\mathscr
A_{p,\sigma}^{\otimes n})$ via
\eqref{Equ:Image dans H0}, the sup-norms of
$v$ relatively to the metrics in $\alpha_n$
are bounded from above by the John norms of
$v$ considered as a section of $L_\sigma$
corresponding to the sup-norms induced by the
norms of $\overline{\mathscr L}$. Thus
Corollary \ref{Cor:Appendix} combined with
\eqref{Equ:integral x+ est volume} implies
that\[\sup_{n}\widehat{\mathrm{vol}}(B^{(p)},\alpha_n)
\geqslant\widehat{\mathrm{vol}}(B^{(p)}).\]
Therefore, by Theorem \ref{The:main theorem},
for any $\varepsilon>0$, there exist certain
integers $p\geqslant 1$ and $n\geqslant 1$
such that
$p^{-d}\widehat{\mathrm{vol}}(\mathscr
A_p,\alpha_n)\geqslant
\mathrm{vol}(\overline{\mathscr
L})-\varepsilon$. Here $(\mathscr
A_p,\alpha_n)$ is rather nef and big since it
is generated by effective global sections.
However, a slight perturbation of
$\overline{\mathscr L}$ permits to conclude.
\end{proof}

\section{Approximating subalgebras}
\label{SubSec:un exemple}

We keep the notation in \S
\ref{Sec:arithmetic fujita}. In this section,
we show that if a positive finite generated
subalgebra of $\overline{\mathscr B}$
approximates well the arithmetic volume of
$\overline{\mathscr L}$, then it also
approximates well the asymptotic measure of
$\overline{\mathscr L}$ truncated at $0$.

Let $p\geqslant 1$ be an integer. Assume that
$\overline{\mathscr L}^{\otimes p}$ is
decomposed as $\overline{\mathscr
A}\otimes\overline{\mathscr M}$, where
$\overline{\mathscr A}$ is arithmetically
ample and $\overline{\mathscr M}$ has a
non-zero effective section $s$. Through the
section $s$ we may consider the section
algebra $\bigoplus_{n\geqslant
0}H^0(X,\mathscr A_K^{\otimes n})$ as a
graded sub-$K$-algebra of $B$. As
$\overline{\mathscr A}$ is ample, for
sufficiently large $n$, one has
$H^0(X,\mathscr A_K^{\otimes
n})\subset\mathcal F^{M}_0(B_{np})$.

\begin{prop}\label{Pro:comparaison de vol}
Let $p\geqslant 1$ be an integer and $S$ be a
graded subalgebra of $B$ generated by a
subspace of $B_p$. For any integer
$n\geqslant 1$, let $\overline{\mathscr S}_n$
be the saturated sub-$\mathcal O_K$-module of
$B_n$, equipped with induced metrics, and
such that $\mathscr S_{n,K}=S_n$; let
$\nu_{\overline{\mathscr S}_n}$ be the
measure associated to the minimum filtration
of $\overline{\mathscr S}_n$. Denote by $\nu$
the vague limit of the measure sequence
$(T_{\frac 1{np}}\nu_{\overline{\mathscr
S}_{np}})_{n\geqslant 1}$. The for any
$x\in\mathbb R$, one has
\begin{equation}\label{Equ:comparaison des mesure}\mathrm{vol}(S)\nu([x,+\infty[)\leqslant\mathrm{vol}(L
)\nu_{\overline{\mathscr
L}}([x,+\infty[),\end{equation} where
$\nu_{\overline{\mathscr L}}$ is the vague
limit of $(T_{\frac
1n}\nu_{\overline{\mathscr B}_n})_{n\geqslant
1 }$, $\nu_{\overline{\mathscr B}_n}$ being
the measure associated to the minimum
filtration of $\overline{\mathscr B}_n $.
Furthermore, if $e_{\min}(\overline{\mathscr
S}_{np})>0$ holds for sufficiently large $n$,
then
\begin{equation}\label{Equ:volume de S plus petit que vol de L}
\mathrm{vol}(S)\leqslant\mathrm{vol}(L)
\nu_{\overline{\mathscr L}}([0,+\infty[)
\end{equation}
\end{prop}
\begin{proof}
 For any $x\in\mathbb R$, one
has
\[\rang( S_{np})\nu_{\overline{\mathscr S}_{np}}([npx,+\infty[)
\leqslant\rang(B_{np})\nu_{\overline{\mathscr
B}_{np}}([npx,+\infty[),\] since these two
quantities are respectively the ranks of
$\mathcal F_{npx}^{M}S_{np}$ and $\mathcal
F_{npx}^{M}B_{np}$. By passing
$n\rightarrow+\infty$, one obtains that, for
any $x\in\mathbb R$,
\[\mathrm{vol}(S)\nu([x,+\infty[)\leqslant\mathrm{vol}(L
)\nu_{\overline{\mathscr L}}([x,+\infty[).\]
Since the positivity condition on last minima
implies that $\nu([0,+\infty[)=1$, one
obtains \eqref{Equ:volume de S plus petit que
vol de L}.
\end{proof}

\begin{coro}With the notation of Proposition \ref{Pro:comparaison de
vol}, assume that
\[\widehat{\mathrm{vol}}(S):=\lim_{n\rightarrow\infty}\frac{\widehat{\deg}(\overline{\mathscr S}_{np})}
{(np)^d/d!}\geqslant
(1-\varepsilon)\widehat{\mathrm{vol}}(\overline{\mathscr
L }),\] where $0<\varepsilon<1$ is a
constant. Then one has
\begin{equation}\label{Equ:approximation de mesure}
0\leqslant\delta_Kd\int_0^{+\infty}
\Big[\mathrm{vol}(L )\nu_{\overline{\mathscr
L}}([x,+\infty[)-\mathrm{vol}(S)\nu([x,+\infty[)\Big]\,\mathrm{d}x\leqslant
\varepsilon\widehat{\mathrm{vol}}(\overline{\mathscr
L }).\end{equation}
\end{coro}
\begin{proof}
By \eqref{Equ:integral x+ est volume}, one
obtains
\[\widehat{\mathrm{vol}}(\overline{\mathscr L})=
\delta_Kd\,\mathrm{vol}(L)\int_{\mathbb
R}\max\{t,0\}\nu_{\overline{\mathscr
L}}(\mathrm{d}t)
=\delta_Kd\,\mathrm{vol}(L)\int_{0}^{+\infty}
\nu_{\overline{\mathscr
L}}([x,+\infty[)\,\mathrm{d}x.\] Similarly,
\[\widehat{\mathrm{vol}}(S)=\delta_Kd\,\mathrm{vol}(S)\int_{0}^{+\infty}
\nu([x,+\infty[)\,\mathrm{d}x.\] Hence the
inequality \eqref{Equ:approximation de
mesure} results from \eqref{Equ:comparaison
des mesure}.
\end{proof}

\appendix
\section{Comparison of filtered graded algebras}

Let $B=\bigoplus_{n\geqslant 0}$ be an
integral graded algebra of finite type over
an infinite field $K$ and $f:\mathbb
N\rightarrow\mathbb R$ be a function such
that
$\displaystyle\lim_{n\rightarrow+\infty}f(n)/n=0$.
We suppose that $B_1\neq 0$ and that $B$ is
generated as $K$-algebra by $B_1$. Assume
that each $B_n$ is equipped with an $\mathbb
R$-filtration $\mathcal F$ such that $B$
becomes a $f$-quasi-filtered graded algebra
(see \S \ref{Subsec:converges de mesure} for
definition). For all integers $m,n\geqslant
0$, let $\mathcal F^{(m)}$ be another
$\mathbb R$-filtration on $B_n$ such that $B$
equipped with $\mathbb R$-filtrations
$\mathcal F^{(m)}$ is $f$-quasi-filtered. For
all integers $m,n\geqslant 1$, let
$\nu_n=T_{\frac 1n} \nu_{(B_n,\mathcal F)}$
and $\nu_n^{(m)}=T_{\frac
1n}\nu_{(B_n,\mathcal F^{(m)} )}$. Assume in
addition that $\lambda_{\max}(B_n,\mathcal
F)\ll n$ and $\lambda_{\max}(B_n,\mathcal
F^{(m)})\ll_mn$. By Theorem
\ref{Thm:convergence de mesures}, the
sequence of measures
$(\nu_n^{(m)})_{n\geqslant 1}$ (resp.
$(\nu_n)_{n\geqslant 1}$) converges vaguely
to a Borel probability which we denote by
$\nu^{(m)}$ (resp. $\nu$).

The purpose of this section is to establish
the following comparison result:
\begin{prop}\label{Pro:A_1}
Let $\varphi$ be an increasing, concave and
Lipschitz function on $\mathbb R$. Assume
that, for any $m\geqslant 1$ and any
$t\in\mathbb R$, one has $\mathcal
F_tB_m\subset\mathcal F_t^{(m)}B_m$, then
\begin{equation}\label{Equ:comparaison des integrale}\limsup_{m\rightarrow+\infty}\int_{\mathbb
R} \varphi\,\mathrm{d}\nu^{(m)}\geqslant
\int_{\mathbb R
}\varphi\,\mathrm{d}\nu.\end{equation}
\end{prop}
\begin{proof}
By Noether's normalization theorem, there
exists a graded subalgebra $A$ of $B$ such
that $A$ is isomorphic to the polynomial
algebra generated by $A_1$. We still use
$\mathcal F^{(m)} $ (resp. $\mathcal F$) to
denote the induced filtrations on $A$. Let
$\widetilde{\nu}_n=T_{\frac
1n}\nu_{(A_n,\mathcal F)}$ and
$\widetilde{\nu}_{n}^{(m)}=T_{\frac
1n}\nu_{(A_n,\mathcal F^{(m)})}$. For any
integer $m\geqslant 1$ and any $t\in\mathbb
R$, one still has $\mathcal
F_tA_m\subset\mathcal F_t^{(m)}A_m$.
Furthermore, by \cite[Proof of Theorem
3.4.3]{Chen08}, the sequence of measures
$(\widetilde{\nu}_n^{(m)})_{n\geqslant 1}$
(resp. $(\widetilde{\nu}_n)_{n\geqslant 1}$)
converges vaguely to $\nu^{(m)}$ (resp.
$\nu$). Therefore, we may suppose that $B=A$
is a polynomial algebra. In this case,
\cite[Proposition 3.3.3]{Chen08} implies that
\[nm\int
\varphi\,\mathrm{d}\nu_{nm}^{(m)} \geqslant
nm\int
\varphi\,\mathrm{d}\nu_{m}^{(m)}-n\|\varphi\|_{\mathrm{Lip}}
f(m)\geqslant nm\int
\varphi\,\mathrm{d}\nu_m-n\|\varphi\|_{\mathrm{Lip}}
f(m).\] since $\nu_m^{(m)}\succ\nu_m$. By
passing $n\rightarrow\infty$, we obtain
\[\int\varphi\,\mathrm{d}\nu^{(m)}\geqslant\int\varphi\,
\mathrm{d}\nu_m-n\|\varphi\|_{\mathrm{Lip}}\frac{f(m)}{m},\]
which implies \eqref{Equ:comparaison des
integrale}.
\end{proof}

In the following, we apply Proposition
\ref{Pro:A_1} to study algebras in metrized
vector bundles. From now on, $K$ denotes a
number field. We assume given an $\mathcal
O_K$-algebra $\mathscr
B=\bigoplus_{n\geqslant 0}\mathscr B_n$,
generated by $\mathscr B_1$, and such that
\begin{enumerate}[1)]
\item each $\mathscr B_n$ is a projective $\mathcal
O_K$-module of finite type;
\item for any integer $n\geqslant 0$, $B_n=\mathscr
B_{n,K}$;
\item the algebra structure of $\mathscr B$
is compatible to that of $B$.
\end{enumerate}
For each integer $n\geqslant 1$, assume that
$g$ is a family of norms on $\mathscr B_n$
such that $(\mathscr B_n,g)$ becomes a
metrized vector bundle on $\Spec\mathcal
O_K$. For all integers $n\geqslant 1$ and
$m\geqslant 1$, let $g^{(m)}$ be another
metric structure on $\mathscr B_n$ such that
$(\mathscr B_n,g^{(m)})$ is also a metrized
vector bundle on $\Spec\mathcal O_K$. Denote
by $\nu_{(B_n,g)}$ and $\nu_{(B_n,g^{(m)})}$
be the measure associated to the minimum
filtration of $(B_n,g)$ and of
$(B_n,g^{(m)})$, respectively.

\begin{coro}\label{Cor:Appendix}
With the notation above, assume in addition
that
\begin{enumerate}[1)]
\item $(\mathscr B,g)$ and all $(\mathscr
B,g^{(m)})$ verify the three conditions in
Remark \ref{Rem:condition de rumely};
\item the identity homomorphism $\mathrm{Id}:(\mathscr B_m,g)\rightarrow
(\mathscr B_m,g^{(m)})$ is effective (see \S
\ref{SubSec:slope method}).
\end{enumerate}
Let $\nu$ and $\nu^{(m)}$ be respectively the
limit measure of $(T_{\frac
1n}\nu_{(B_n,g)})_{n\geqslant 1}$ and
$(T_{\frac 1n}\nu_{(\mathscr
B_n,g^{(m)})})_{n\geqslant 1}$. Then for any
increasing, concave and Lipschtiz function
$\varphi$ on $\mathbb R$, one has
\[\limsup_{m\rightarrow\infty}\int_{\mathbb R}
\varphi\,\mathrm{d}\nu^{(m)}\geqslant\int_{\mathbb
R }\varphi\,\mathrm{d}\nu.\] In particular,
if
${\displaystyle\liminf_{n\rightarrow\infty}}\frac
1n e_{\min}(\mathscr B_n,g)\geqslant 0$ and
if
${\displaystyle\liminf_{m\rightarrow\infty}\liminf_{n\rightarrow\infty}}
\frac 1n e_{\min}(\mathscr
B_n,g^{(m)})\geqslant 0$, then
\begin{equation}\label{Equ:comparaison de volume}
\limsup_{m\rightarrow\infty}\int_{\mathbb
R} \max\{x,0\}\,\nu^{(m)}(\mathrm{d}x)
\geqslant \int_{\mathbb
R}\max\{x,0\}\,\nu(\mathrm{d}x).\end{equation}
\end{coro}
\begin{proof}
The first assertion is a direct consequence
of Proposition \ref{Pro:A_1}. In particular,
one has
\[\limsup_{m\rightarrow\infty}\int_{\mathbb R}
\max\{x,0\}\,\nu^{(m)}(\mathrm{d}x) \geqslant
\int_{\mathbb
R}\max\{x,0\}\,\nu(\mathrm{d}x).\] The
hypoth\`ese
${\displaystyle\liminf_{n\rightarrow\infty}}\frac
1n e_{\min}(\mathscr B_n,g)\geqslant 0$
implies that the support of $\nu$ is bounded
from below by $0$, so $\int_{\mathbb
R}\max\{x,0\}\,\nu(\mathrm{d}x)=\int_{\mathbb
R}x\,\nu(\mathrm{d}x)$. For any integer
$m\geqslant 1$, let
$a_m={\displaystyle\liminf_{n\rightarrow\infty}}\frac
1n e_{\min}(\mathscr B_n,g_m)$ and
$b_m=\min(a_m,0)$. One has $\int_{\mathbb
R}x\,\nu^{(m)}(\mathrm{d}x)=\int_{\mathbb R}
\max(x,b_m)\,\nu^{(m)}(\mathrm{d}x)$. Note
that
\[\bigg|\int_{\mathbb R}
\max(x,b_m)\,\nu^{(m)}(\mathrm{d}x)-
\int_{\mathbb R}
\max(x,0)\,\nu^{(m)}(\mathrm{d}x)\bigg|\leqslant
b_m,
\]
which converges to $0$ when
$m\rightarrow\infty$, so we obtain
\eqref{Equ:comparaison de volume}.
\end{proof}

\backmatter
\bibliography{chen}
\bibliographystyle{smfplain}

\end{document}